\documentclass[reqno, 11pt]{amsart}%
\usepackage{amssymb}
\usepackage{amsaddr}
\usepackage[nesting]{hyperref}
\usepackage[pdftex]{graphicx}
\usepackage{listings}
\usepackage{multirow}
\usepackage{placeins}
\usepackage{color}
\usepackage{subfigure}
\usepackage{lscape}
\usepackage{dsfont}
\usepackage{amsmath}
\usepackage{amsfonts}
\usepackage{tikz}
\usepackage{verbatim}
\usepackage{accents} 
\usepackage{todonotes}
\usepackage{mathtools}
\usepackage{bbm}
\usepackage{caption}
\usepackage{natbib} 
\usepackage{float}

\newcommand{\BlackBoxes}{\global\overfullrule5pt}

\BlackBoxes

\providecommand{\U}[1]{\protect\rule{.1in}{.1in}}

\newcommand{\R}{\mathbb{R}} 
\newcommand{\N}{\mathbb{N}} 

\newcommand{\PP}{\mathbb{P}}
\newcommand{\EE}{\mathbb{E}}
\newcommand{\E}{\mathbb{E}}

\newtheorem{theorem}{Theorem}

\newtheorem{lemma}[theorem]{Lemma}

\theoremstyle{definition}
\newtheorem{example}[theorem]{Example}
\newtheorem{remark}[theorem]{Remark}

\numberwithin{equation}{section} \numberwithin{theorem}{section}
\def\0{\kern0pt\-\nobreak\hskip0pt\relax}

\makeatletter
\AtBeginDocument{ \def\@serieslogo{ \vbox to\headheight{ \parindent\z@ \fontsize{6}{7\p@}\selectfont
\vss}}}

\def\makeoverbar#1#2#3#4#5#6#7{ \setbox0=\hbox{$\m@th#2\mkern#5mu{{}#3{}}\mkern#6mu$} 
\setbox1=\null \dimen@=#4\fontdimen8#13 \dimen@=3.5\dimen@
\advance\dimen@ by \ht0 \dimen@=-#7\dimen@ \advance\dimen@ by \wd0
\ht1=\ht0 \dp1=\dp0 \wd1=\dimen@
\dimen@=\fontdimen8#13 \fontdimen8#13=#4\fontdimen8#13
\rlap{\hbox to \wd0{$\m@th\hss#2{\overline{\box1}}\mkern#5mu$}}
\fontdimen8#13=\dimen@}
\def\mylabel#1#2{{\def\@currentlabel{#2}\label{#1}}}
\makeatother
\begin{document}
\title[  ]{Continuous-time mean field Markov decision models}
\author[N. \smash{B\"auerle}]{Nicole B\"auerle${}^*$}
\address[N. B\"auerle]{Department of Mathematics,
Karlsruhe Institute of Technology (KIT),  Karlsruhe, Germany}

\email{\href{mailto:nicole.baeuerle@kit.edu}{nicole.baeuerle@kit.edu}}

\author[S. \smash{H\"ofer}]{Sebastian H\"ofer${}^*$}
\address[S. H\"ofer]{Department of Mathematics,
Karlsruhe Institute of Technology (KIT),  Karlsruhe, Germany}

\email{\href{mailto:sebastian.hoefer@kit.edu} {sebastian.hoefer@kit.edu}}


\begin{abstract}
We consider a finite number of $N$ statistically equal agents,  each moving on a finite set of states according to a continuous-time Markov Decision Process (MDP). Transition intensities of the agents and generated rewards depend not only on the state and action of the agent itself, but also on the states of the other agents as well as the chosen action. Interactions like this are typical for a wide range of models in e.g.\  biology, epidemics, finance, social science and queueing systems among others. The aim is to maximize the expected discounted reward of the system, i.e. the agents have to cooperate as a team. Computationally this is a difficult task when $N$ is large. Thus, we consider the limit for $N\to\infty.$ In contrast to other papers we treat this problem from an MDP perspective.  This has the advantage that we need less regularity assumptions in order to construct asymptotically optimal strategies than using viscosity solutions of HJB equations. We also show that the convergence rate is $1/\sqrt{N}$. We show how to apply our results using two examples: a machine replacement problem and a problem from epidemics. We also show that optimal feedback policies from the limiting problem are not necessarily asymptotically optimal.

\end{abstract}
\maketitle


\makeatletter \providecommand\@dotsep{5} \makeatother



\vspace{0.5cm}
\begin{minipage}{14cm}
{\small
\begin{description}
\item[\rm \textsc{ Key words}]
{\small  Continuous-time Markov decision process, Mean field problem, Process limits, Pontryagin's maximum principle}
\item[\rm \textsc{AMS subject classifications}] 
{\small 90C40,60G55,60F17}

\end{description}
}
\end{minipage}

\section{Introduction}
We consider a finite number of $N$ statistically equal agents,  each moving on a finite set of states according to a continuous-time Markov Decision Process. Transition intensities of the agents and generated rewards can be controlled and depend not only on the state and action of the agent itself, but also on the states of the other agents. Interactions like this are typical for a wide range of models in e.g.\  biology, epidemics, finance, social science and queueing systems among others. The aim is to maximize the expected discounted reward of the system, i.e. the agents have to cooperate as a team. This can be implemented by a central controller who is able to observe the whole system and assigns actions to the agents. Though this system itself can be formulated as a continuous-time Markov Decision Process, the established solution procedures are not really practical since the state space of the system is complicated and of high cardinality. Thus, we consider the limit $N\to\infty$ when the number of agents tends to infinity and analyze the connection between the limiting optimization problem which is a deterministic control problem and the $N$ agents problem. 

Investigations like this are well-known under the name {\em Mean-field approximation}, because the mean dynamics of the agents can be approximated by differential equations for a measure-valued state process. This is inspired by statistical mechanics and can be done for different classes of stochastic processes for the agents. In our paper we restrict our investigation to continuous-time Markov chains (CTMC). Earlier, more practical studies in this spirit with CTMC, but {\em without control} are e.g.\ \cite{bortolussi2013continuous,kolesnichenko2014applying} 
which consider illustrating examples to discuss how the mean-field method is used in different application areas. The convergence proof there is based on the law of large numbers for centred Poisson processes, see also \cite{kurtz1970solutions}. \cite{ball2006asymptotic} look at so-called reaction networks which are chemical systems involving multiple reactions and
chemical species. They take approximations of multiscale nature into account and show that 'slow' components can be approximated by a deterministic equation. \cite{darling2008differential} formulate some simple conditions under which a CTMC may be approximated by the solution to a differential equation, with quantifiable error probabilities. They give different applications. \cite{aspirot2011fluid} explore models proposed for the analysis of BitTorrent P2P systems and provide the arguments to justify the passage from the stochastic process, under adequate scaling, to a fluid approximation driven by a differential equation. A more recent application is given in \cite{kyprianou2022replicator} where 
a multi-type analogue of Kingman’s coalescent as a death chain is considered. The aim is to characterize  the behaviour of the replicator coalescent as it is started  from an initial population that is arbitrarily large. This leads to a differential equation called the replicator equation. A similar control model as ours is considered in \cite{cecchin2021finite}. However, there the author uses a finite time horizon and solves the problem with HJB equations. This requires a considerable technical overhead like viscosity solutions and more assumptions on the model data like Lipschitz properties which we do not need here.

A related topic are fluid models.  Fluid models have been introduced in queueing network theory since  there is a close connection  between the stability of the stochastic network and the corresponding  fluid model, \cite{meyn1997stability}. They appear under 'fluid scaling' where time in the CTMC for the stochastic queueing network is accelerated by a factor $N$ and the state  is compressed by factor $1/ N.$ Fluid models have also been used to approximate the optimal control in these networks, see e.g.\ \cite{avram1995fluid,weiss1996optimal,bauerle2000asymptotic,bauerle2002optimal,vcudina2011asymptotically}. In \cite{yin2012continuous} different scales of time  are treated for the approximation and some components may be replaced by differential equations. But there is no mean-field interaction in all of these fluid models.

There are also investigations about controlled mean-field Markov decision processes and their limits in discrete time. An early paper is \cite{gast2012mean} where the mean-field limit for increasing number of agents is considered in a model where only  the central controller is allowed to choose one action. However, in order to get a continuous limit the authors have to interpolate and rescale the original discrete-time processes. This implies the necessity for assumptions on the transition probabilities. The authors show the convergence of the scaled value functions and derive asymptotically optimal strategies.
The recent papers \cite{carmona2019modelx,motte2022mean,motte2023quantitative,bauerle2023mean} discuss the convergence of value functions and asymptotically optimal policies in discrete time. In contrast to our paper they allow a common noise. The limit problem is then a controlled stochastic process in discrete-time. 

Another strand  of literature considers continuous-time mean-field {\em games} on a finite number of states \cite{gomes2013continuous,basna20141,bayraktar2018analysis,cecchin2020probabilistic,belak2021continuous}. These papers  among others consider the construction of asymptotically optimal Nash-equilibria from a limiting equation. The exception is \cite{cecchin2020probabilistic} where it is shown that any solution of the limiting game can be approximated by $\epsilon_N$-Nash equilibria in the $N$ player game. However, all these papers deal with the convergence of the HJB equations which appear in the $N$ player game to a limiting equation, called the Master equation (\cite{cardaliaguet2019master}) which  is a deterministic PDE for the value function. This approach needs sufficient regularity of the value functions and many assumptions. \cite{belak2021continuous} consider the problem with common noise and reduce the mean field equilibrium to a system of forward-backward systems of (random) ordinary differential equations.

The {\em contribution of our paper} is first to establish and investigate the limit of the controlled continuous-time Markov decision processes. In contrast to previous literature which works with the HJB equation this point of view requires less assumptions e.g.\ we do not need Lipschitz conditions on the model data.  
Second, we are also able to construct an asymptotically optimal strategy for the $N$ agents model. Our model is general, has only a few, easy to check assumptions and allows for various applications. The advantage of our limiting optimization problem is that we can apply Pontryagin's maximum principle easily which is often more practical than deterministic dynamic programming. Further, we show that an optimal feedback policy in the deterministic problem does not necessarily imply an asymptotically optimal policy for the $N$ agents problems. Third, we obtain a convergence rate in a straightforward way. Fourth, we can consider finite and infinite time horizon at the same time. There is essentially no difference. We restrict the presentation mainly to the infinite time horizon.

Our paper is organized as follows: In the next section we introduce our $N$ agents continuous-time Markov decision process. The aim is to maximize the expected discounted reward of the system. In Section \ref{sec:limitmodel} we introduce a measure-valued simplification which is due to the symmetry properties of the problem and which reduces the cardinality of the state space. The convergence theorem if the number of agents tends to infinity can be found in Section \ref{sec:convergenceproof}. It is essentially based on martingale convergence arguments. In Section \ref{sec:F} we construct a sequence of asymptotically optimal strategies from the limiting model for the $N$ agents model. We also show that different implementations may be possible and that the rate of convergence is at most $1/\sqrt{N}$. Finally in Section \ref{sec:appl} we discuss three applications. The first one is a machine replacement problem when we have many machines, see e.g. \cite{thompson1968optimal}. The second one is the spreading of malware which is based on the classical SIR model for spreading infections, \cite{khouzani2012maximum,gast2012mean}. The last example shows that one has to be careful with feedback policies.

\section{The $N$ agents continuous-time Markov Decision Process}\label{sec:model1}

We consider a finite number of $N$ statistically equal agents,  each moving on a finite set of states $S$ according to a continuous-time Markov Decision Process.  The vector $\mathbf{x}_t = (x_t^1,...,x_t^N)\in S^N$ describes the state of the system at time $t\in [0,\infty)$, where $x_t^k$ is the state of agent $k=1,\dots,N$. The action space of one agent is a compact Borel set $A$. The action space of the system is accordingly $A^N$. We denote an action of the system by $\mathbf{a} = (a^1,...,a^N)\in A^N$ where $a^k$ is the action chosen by agent $k=1,\ldots,N$.\\
Let $D(i)\subset A$ be the set of actions available for an agent in state $i\in S$ which we again assume to be compact. Then the set of admissible actions for the system in state $\mathbf{x} \in S^N$ is given by $\mathbf{D}(\mathbf{x}):= D(x^1)\times \dots \times D(x^N)\subset A^N$. The set of admissible state-action combinations for one agent is denoted by $D:= \{(i, a) \in S\times A \mid  a \in D(i), \ \forall \ i\in S\}$.
\\\\
For the construction of the system state process we follow the notation of \cite{piunovskiy2020continuous}. The state process of the system is defined on the measurable space $(\Omega,\mathcal{F}):= \big((S^N\times \R_+)^\infty,\mathcal{B}((S^N\times \R_+)^\infty)\big).$ We denote an element of $\Omega$ by $\omega=(\textbf{x}_0,t_1,\textbf{x}_1,t_2,...)$. Now define
\begin{align*}
   \tilde{\mathbf{X}}_n&:\Omega \to S^N,\quad  \tilde{\mathbf{X}}_n(\omega) = \textbf{x}_n,\quad n\in \N_0,\\
    \tau_n&:\Omega \to \R_+,\quad \  \tau_n(\omega) = t_n,\quad \  n\in \N,\\
    T_n &:= \sum_{k=1}^n \tau_k, \quad T_0 := 0.
\end{align*}
The controlled state process of the system is then given by
\begin{equation}
    \nonumber
    \textbf{X}_t := \sum_{n\in \N_0} \mathds{1}_{\{T_n\leq t< T_{n+1}\}} \tilde{\mathbf{X}}_n,\qquad t\in [0,\infty).
\end{equation}
The construction of the process can be interpreted as follows: The random variables $\tau_n$ describe the sojourn times of the system in states $\tilde{\mathbf{X}}_{n-1}$. Based on the sojourn times, $T_n$ describes the time of the $n$-th jump of the process and $\tilde{\mathbf{X}}_n$ the state of the process on the 
interval $[T_n,T_{n+1})$. By construction the continuous-time state process $(\mathbf{X}_t)$ has piecewise constant càdlàg-paths and the embedded discrete-time process is $(\tilde{\mathbf{X}}_n)$.\\\\
The system is controlled by policies. W.l.o.g.\ we restrict here to Markovian stationary policies. Further, we allow for randomized decisions, i.e.\ each agent can choose a probability distribution on $A$ as its action. Hence a policy for the system is given by a collection of $N$ stochastic kernels $\pi(d\mathbf{a}\mid \mathbf{x}) = (\pi^k(da\mid \mathbf{x}))_{k=1,...,N}$, where
\begin{equation*}
    \pi^k:S^N\times \mathcal{B}(A) \to [0,1], \quad (\mathbf{x},\mathcal{A}) \mapsto \pi^k(\mathcal{A}\mid \mathbf{x})\qquad \text{(kernel for agent $k$)}.
\end{equation*}
$\pi^k(\mathcal{A}\mid \mathbf{x})$ is the stochastic kernel (it is here considered as a relaxed control) with which agent $k$ chooses an action, given the state $\mathbf{x}$ of the system. Naturally, it should hold that the kernel is concentrated on admissible actions, i.e. $\pi^k(D(x^k) \mid \mathbf{x})= 1$ for all agents $k=1,...,N$.\\
The action process is thus defined by 
\begin{equation*}
    \pi_t := \sum_{n\in \N_0} \mathds{1}_{\{T_n< t\leq T_{n+1}\}} \pi(\cdot\mid \tilde{\mathbf{X}}_n),\qquad t\in [0,\infty).
\end{equation*}
In contrast to the state process, the action process has piecewise constant càglàd-paths. This means that a new decision can only be taken after a change of state has already occurred. The general theory on continuous-time Markov decision processes states that the optimal policy can be found among the piecewise constant, deterministic, stationary policies. In particular, varying the action continuously on the interval $[T_n,T_{n+1})$ does not increase the value of the problem. Also randomization does not increase the value, but in view of the sections to come, we already allowed for randomization (relaxation)  here. \\\\
To prepare the description of the transition mechanism in our model, we define the empirical distribution of the agents over the states, i.e.
\begin{equation}
    \nonumber
    \mu[\mathbf{x}] := \frac{1}{N} \sum_{k=1}^N \delta_{x^k}. 
\end{equation}
where $\delta_{x^k}$ is the Dirac measure in point $x^k$. The transition intensities for one agent are given by a signed kernel 
\begin{equation*}
    q:S\times A\times \PP(S)\times \mathcal{P}(S) \to \R,\quad (i,a,\mu,\Gamma)\mapsto q(\Gamma\mid i,a,\mu) = \sum_{j\in \Gamma} q(\{j\}\mid i,a,\mu).
\end{equation*}
Here $ \PP(S) $ is the set of all probability distributions on $S$ and $ \mathcal{P}(S)$ is the power set of $S.$
Note that the transition of an agent depends not only on its own state and action, but also on the empirical distribution of all agents over the states.\\\\
We make the following assumptions on $q$:
 \begin{itemize}
     \item[(Q1)] $q(\{j\}|i,a,\mu)\ge 0$ for all $i,j\in S,\ j\neq i, \ a \in D(i), \mu\in \PP(S).$
     \item[(Q2)] $\sum_{j\in S} q(\{j\}|i,a,\mu)=0$ for all $(i,a)\in D, \ \mu\in \PP(S).$
     \item[(Q3)] $\sup_{i,a,j,\mu} |q(\{j\}|i,a,\mu)|=: q_{max}<\infty.$
     \item[(Q4)] $\mu \mapsto q(\{j\}|i,a,\mu)$ is continuous w.r.t.\ weak convergence for all $i,j\in S,\ a\in D(i).$
     \item[(Q5)]  $a \mapsto q(\{j\}|i,a,\mu)$ is continuous for all $i,j\in S,\ \mu\in \PP(S).$
 \end{itemize}
Note that (Q3) follows from (Q4) and (Q5), but since it is important we list it here. Based on the transition intensities for one agent, the transition intensities of the system are given by
 \begin{equation}\label{eq:Qsystem}
     q(\{(x^1,\ldots,x^{k-1},j,x^{k+1},\ldots x^N)\} |\mathbf{x}, \mathbf{a} ) := q(\{j\}|x^k,a^k,\mu[\mathbf{x}])
 \end{equation} 
for all $(\mathbf{x}, \mathbf{a}) \in \mathbf{D}(\mathbf{x}), j\in S, j\neq x^k$ and 
$$ q(\{{\mathbf{x}}\}| \mathbf{x}, \mathbf{a}) := \sum_{k=1}^N q(\{ x^k\}|x^k,a^k,\mu[\mathbf{x}]).$$
All other intensities are $=0.$ The intensity in \eqref{eq:Qsystem} describes the transition of agent $k$ from state $x^k\in S$ to state $j\in S$, while all other agents stay in their current state. Since only one agent can change its state at a time, this definition is sufficient to describe the transition mechanism of the system.\\
Further we set (in a relaxed sense) for a decision rule $\pi^k(da|\mathbf{x})$
$$ q(\{(x^1,\ldots,x^{k-1},j,x^{k+1},\ldots ,x^N)\} |\mathbf{x}, \pi ) = \int_A q(\{j\}|x^k,a,\mu[\mathbf{x}])\pi^k(da|\mathbf{x}).$$
Note that in a certain sense there is abuse of notation here since we use the letter $q$ both for the agent transition intensity and for the system transition intensity. It should always be clear from the context which one is meant.\\\\
The probability measure of the $N$ agent process is now given by the following transition kernels
\begin{align*}
    \PP^\pi(\tau_n\le t, \tilde{\mathbf{X}}_n \in B| \tilde{\mathbf{X}}_{n-1}) =
      \int_0^t q(B| \tilde{\mathbf{X}}_{n-1},\pi) e^{s \cdot q (\{\tilde{\mathbf{X}}_{n-1}\}| \tilde{\mathbf{X}}_{n-1},\pi)}ds
\end{align*}
for all $t\geq0$ and $B\in \mathcal{P}(S^N).$ In particular, the sojourn times $\tau_n$ are exponentially distributed with parameter $-q (\{\tilde{\mathbf{X}}_{n-1}\}| \tilde{\mathbf{X}}_{n-1},\pi)$ respectively. Note that by using this construction, the probability measure depends on the chosen policy. This construction is more convenient when the transition intensities are given. In case the system is described by transition functions and external noise it is easier to use a common probability space which does not depend on the policy. Of course these two points of view are equivalent.
\\\\
Returning to the model's control mechanism, keep in mind that the policy of an agent $\pi^k(da\mid \mathbf{x})$ is allowed to depend on the state of the whole system, i.e. we assume that each agent has information about the position of all other agents. Therefore, we can interpret our model as a centralized control problem, where all information is collected and shared by a central controller.\\
The goal of the central controller is to maximize the social reward of the system. In order to implement this, we introduce the (stationary) reward function for one agent as 
\begin{equation}
    \nonumber
    r:D\times \PP(S)\to \R,\qquad (i,a,\mu) \mapsto r(i,a,\mu),
\end{equation}
which does not only depend on the state and action of the agent, but also on the empirical distribution of the system. We make the following assumptions on the reward function:
\begin{itemize}
    \item[(R1)]  For all $(i,a)\in D$  the function $\mu \mapsto \mu(i) r(i,a,\mu)$ is continuous w.r.t. weak convergence.
    \item[(R2)]  For all $i\in S$ and $\mu \in \PP(S)$  the function $a \mapsto r(i,a,\mu)$ is continuous.
\end{itemize}
Since the set of admissible actions $D(i)$ is compact, (R1) and (R2) imply that the following expression is bounded:
  \begin{equation}\label{eq:rbounded}
         \sup_{(i,a)\in D, \ \mu\in \PP(S)} |\mu(i) r(i,a,\mu)|<\infty.
    \end{equation}
The (social) reward of the system is the average of the agents' rewards
\begin{equation}
    \label{eq:reward1}
    r(\mathbf{x},\mathbf{a}):= \frac{1}{N} \sum_{k=1}^N r(x^k,a^k, \mu[\mathbf{x}]),
\end{equation}
or, in a relaxed sense for a decision rule $\pi^k(da\mid \mathbf{x})$
\begin{equation}
    \nonumber
    r(\mathbf{x},\pi):= \frac{1}{N} \sum_{k=1}^N \int_A r(x^k,a, \mu[\mathbf{x}]) \pi^k(da\mid \mathbf{x}).
\end{equation}
The aim is now to find the \textit{social optimum}, i.e.\ to maximize the joint expected discounted reward of the system over an infinite time horizon. For a policy $\pi$, a discount rate $\beta>0$ and an initial configuration $\mathbf{x}\in S^N$ define the value function
\begin{align}
    \nonumber
    V_\pi(\mathbf{x})&=  \EE_\mathbf{x}^\pi\Big[ \int_0^\infty e^{-\beta t} r(\mathbf{X}_t,\pi_t)dt\Big] \\\label{eq:valueV}
    V(\mathbf{x}) &= \sup_\pi V_\pi(\mathbf{x}).
\end{align}
We are not discussing solution procedures for this optimization problem here since we simplify it in the next section and present asymptotically optimal solution methods in Section \ref{sec:F}.

\section{The measure-valued continuous-time Markov Decision Process}\label{sec:limitmodel}
As $N$ is getting larger, so does the state space $S^N$, which could make the model increasingly complex and impractical to solve. Therefore, we seek for some simplifications. An obvious approach which is common for these kind of models, is to exploit the symmetry of the system by capturing not the state of every single agent, but the relative or empirical distribution of the agents across the $\vert S\vert$ states. \\
Thus, let $\mu_t^N := \mu[\mathbf{X}_t]$ and define as new state space the set of all distributions which are empirical measures of $N$ atoms
\begin{equation}
    \nonumber
    \PP_N(S) := \{\mu\in \PP(S)\mid \mu = \mu[\mathbf{x}], \text{ for } \mathbf{x} \in S^N\}.
\end{equation}
It holds that the new state process $\mu_t^N $ is the same as
$$ \mu_t^N = \sum_{n\in \N_0} \mathds{1}_{\{T_n \le t < T_{n+1}\}}\mu[ \tilde{\mathbf{X}}_n],\qquad t\in [0,\infty).$$
As action space take the $\vert S\vert$-fold Cartesian product $\PP(A)^{\vert S\vert}$ of $\PP(A)$. Hence, an action is given by $\vert S\vert$ probability measures $\alpha(d\mathbf{a}) = (\alpha^i(da))_{i\in S}$ with $\alpha^i(D(i)) = 1$. Hereby the $i$-th component indicates the distribution of the agents' actions in state $i\in S$. The set of admissible state-action combinations of the new model is given by $\hat{D} := \PP_N(S) \times \PP(A)^{\vert S \vert}$. \\\\
For the policies we restrict again to Markovian, stationary policies given by a collection of $\vert S \vert$ stochastic kernels $\hat\pi(d \mathbf{a}|\mu)= (\hat\pi^i(da|\mu))_{i\in S}$, where
\begin{equation*}
    \hat\pi^i:\PP_N(S)\times \mathcal{B}(A) \to [0,1], \quad (\mu,\mathcal{A}) \mapsto \hat\pi^i(\mathcal{A}\mid \mu)\qquad \text{(kernel for state $i$)}.
\end{equation*}
where $\hat\pi^i(D(i)\mid \mu)=1.$
 In what follows we denote $\tilde{\mu}_n^N := \mu[\tilde{\mathbf{X}}_n]$. Then we can express the action process by setting
\begin{eqnarray}\label{eq:actionprocess2}
    \hat\pi_t := \sum_{n\in \N_0} \mathds{1}_{\{T_n < t \le  T_{n+1}\}} \hat\pi(\cdot| \tilde{\mu}_n^N),\qquad t\in [0,\infty).
\end{eqnarray}
The transition intensities of the process $(\mu_t^N)_{t\geq 0}$ are given by
\begin{equation}\label{eq:qtermsystem}q(\{\mu^{i\to j}\}| \mu,\alpha )= N\mu(i) \int_A q(\{j\}| i,a,\mu)  \alpha^i(da), \quad \mu \in \PP_N(S),   \alpha \in \PP(A)^{\vert S \vert}, \end{equation}
with $\mu^{i\to j}:= \mu-\frac1N \delta_{i}+\frac1N \delta_{j}$ for all $i,j\in S, i\neq j$ if $\mu(i)>0.$ This intensity describes the transition of one arbitrary agent in state $i\in S$ to state $j\in S$, while all other agents stay in their current state. Note that the intensity follows from the usual calculations for continuous-time Markov chains, in particular from the fact that if $X,Y$ are independent random variables with $X\sim Exp(\lambda), Y\sim Exp(\nu),$ then $X\wedge Y \sim Exp(\lambda+\nu).$ In the situation in \eqref{eq:qtermsystem} we have $N\mu(i)$ agents in state $i$. Further we set for all $\mu\in \PP_N(S)$ and $\alpha \in \PP(A)^{\vert S\vert}$
$$q(\{\mu\}| \mu,\alpha ):= -\sum_{i, \mu(i)>0}\sum_{j\neq i}q(\{\mu^{i\to j}\}| \mu,\alpha ). $$
All other intensities are zero, since again only one agent can change its state at a time. \\\\
The probability distribution of the measure-valued process under a fixed policy $\hat\pi$ is now given by the following transition kernels
\begin{align*}
    \PP^{\hat\pi}(\tau_n\le t, \tilde{\mu}_n^N \in B| \tilde{\mu}_{n-1}^N) =
      \int_0^t q(B| \tilde{\mu}_{n-1}^N,\hat\pi) e^{s\cdot q (\{\tilde{\mu}_{n-1}^N\}| \tilde{\mu}_{n-1}^N,\hat\pi)}ds
\end{align*}
for all $t\geq0$ and $B\subset \PP_N(S)$ measurable, where the random variables $(\tau_n)$ are the same as before.\\\\
The reward function of the system is derived from the reward for one agent:
\begin{equation}
    \nonumber
    r(\mu,\alpha):= \sum_{i\in S} \int r(i,a,\mu) \alpha^i(da) \mu(i).
    \end{equation}
 In view of \eqref{eq:rbounded}  $r(\mu,\alpha)$ is bounded.
The aim in this model is again to maximize the joint expected discounted reward of the system over an infinite time horizon. For a policy $\hat{\pi}$, a discount rate $\beta>0$ and an initial configuration $\mu\in \PP_N(S)$ define the value function
\begin{eqnarray}\nonumber
V_{\hat\pi}^N(\mu) &=&  \EE_{\mu}^{\hat\pi}\Big[ \int_0^\infty e^{-\beta t} r(\mu_t^N,\hat\pi_t)dt\Big] \\ \label{prob:II}
 V^N(\mu) &=& \sup_{\hat\pi} V_{\hat\pi}^N(\mu).
 \end{eqnarray}
We can now show that both formulations \eqref{eq:valueV} and \eqref{prob:II} are equivalent in the sense that the optimal values are the same. Of course, an optimal policy in the measure-valued setting can directly be implemented in the original problem. The advantage of the measure-valued formulation is the reduction of the cardinality of the state space. Suppose for example that $S=\{0,1\}$, i.e. all agents are either in state 0 or state 1. Then $|S^N|=2^N$ in the original formulation whereas $|\PP_N(S)|=N+1$ in the second formulation. A proof of the next theorem can be found in the appendix.
 
 \begin{theorem}\label{theo:equivalence}
    It holds that $V(\mathbf{x})=V^N(\mu)$ for $\mu=\mu[\mathbf{x}]$ for all $\mathbf{x}\in S^N.$
\end{theorem}

\begin{remark}
    It is possible to extend the previous result to a situation where reward and transition intensity both also depend on the empirical distribution of actions, e.g. \cite{motte2022mean}. However, due to the definition of the Young topology which we use later it is not possible to transfer the convergence results to this setting.
\end{remark}
The problem we have introduced is a classical continuous-time Markov Decision Process and can be solved with the established theory accordingly. Thus, we obtain:

\begin{theorem}\label{theo:sol_classic}
    There exists a continuous function $v:\PP_N(S)\to\R$ satisfying
    $$\beta v(\mu) = \sup_{\alpha\in\PP(A)^{|S|}} \left\{ r(\mu,\alpha) + \int v(\nu) q(d\nu|\mu,\alpha) \right\} $$
    for all $\mu\in \PP_N(S)$ and there exists a maximizer $\hat\pi(\cdot|\mu)$ of the r.h.s.\ such that $v=V^N$ and $\hat\pi$ determines the optimal policy by \eqref{eq:actionprocess2}.
\end{theorem}
Follows from Theorem 4.6, Lemma 4.4 in \cite{guo2009continuous} or Theorem 3.1.2 in \cite{piunovskiy2020continuous}.\\

Theorem \ref{theo:sol_classic} implies a solution method for problem \eqref{prob:II}. It can e.g.\ be solved by value or policy iteration. However, as already discussed, even in this simplified setting, the computation may be inefficient if $N$ is large, since this leads to a large state space. 

\section{Convergence of the state process}\label{sec:convergenceproof}
In this section we discuss the behaviour of the system when the number of agents tends to infinity. In this case we obtain a deterministic limit control model which serves as an asymptotic upper bound for our optimization problem with $N$ agents. Moreover, an optimal control of the limit model can be used to establish a sequence of asymptotically optimal policies for the $N$ agents model.

In what follows we consider $(\mu_t^N)$  as a stochastic element of $D_{\PP_N(S)}[0,\infty)$, the space of c\`adl\`ag paths with values in $\PP_N(S)$  equipped with the Skorokhod $J_1$-topology and metric $d_{J_1}.$ On $\PP_N(S)$ we choose the total variation metric. 

Further, we consider $\hat\pi^i$ as a stochastic element in $\mathcal{R}:= \{\rho:\R_+\to \PP(A)\ |\ \rho \mbox{ measurable}\}$ endowed with the Young topology (cf. \cite{davis2018markov}). It is possible to show that $\mathcal{R}$ is compact and metrizable. Measurability and convergence in $\mathcal{R}$ can be characterized as in Lemma \ref{lem:Young}. These statements follow directly from the fact that the Young topology is the coarsest topology such that the mappings
$$ \rho \mapsto \int_0^\infty \int_A \psi(t,a) \rho_t(da)dt$$
are continuous for all real functions $\psi$ on $\R_+\times A$ where $\psi$ is a Carath\'eodory function, i.e. $\psi$ is continuous in $a$ and measurable in $t$ where $\psi$ is integrable in the sense that $\int_0^\infty \sup_a |\psi(t,a)|dt <\infty.$

\begin{lemma}\label{lem:Young}
    \begin{itemize}
        \item[a)] $\rho:\R_+\to \PP(A)$ is measurable if and only if $\rho$ is a transition probability from $\R_+$ into $A$.
        \item[b)] Let $\rho^n,\rho\in\mathcal{R}.$ $\rho^n\to \rho$ for $n\to\infty$ if and only if 
        $$ \int_0^\infty \int_A \psi(t,a) \rho_t^n(da)dt \to  \int_0^\infty \int_A \psi(t,a) \rho_t(da)dt $$
        for all measurable functions $\psi:\R_+\times A\to\R$ such that $a\mapsto \psi(t,a)$ is continuous for all $t\ge 0$ and $\int_0^\infty \sup_a |\psi(t,a)|dt <\infty.$
    \end{itemize}
\end{lemma}

In a first step we define for $N\in\N$, a fixed policy $\hat\pi^N$ and arbitrary $j\in S$, the one-dimensional process
\begin{align*}
    M_t^N(j)&:=\mu_t^N(j)-\mu_0^N(j)-\int_0^t \sum_{\nu\in \PP_N(S)} (\nu(j)-\mu_s^N(j)) q (\{\nu\}|\mu_s^N,\hat\pi_s)ds.
\end{align*} 
Then $(M_t^N(j))$ are martingales w.r.t.\ the filtration $\mathcal{F}_t^N = \sigma(\mu_s^N,s\le t).$ This follows from the Dynkin formula, see e.g.\ \cite{davis2018markov}, Proposition 14.13. Next we can express the process $(M_t^N(j))$ a bit more explicitly. Note that the difference   $\nu(j)-\mu_s^N(j)$ can either be $-1/N$ if an agent changes from state $j$ to a state $k\neq j$ or it could be $1/N$ if an agent changes from state $i\neq j$  to state $j$. Since by (Q2)
\begin{equation}\label{eq:qconservative} \sum_{k\neq j} \int  q (\{k\}|j,a,\mu_s^N) \hat\pi^{N,j}_s(da) = - \int  q (\{j\}|j,a,\mu_s^N) \hat\pi_s^{N,j}(da) \end{equation}
we obtain by inserting the intensity \eqref{eq:qtermsystem} and by using \eqref{eq:qconservative}
\begin{align}\nonumber
    M_t^N(j)=&\mu_t^N(j)-\mu_0^N(j)-\int_0^t \sum_{k\neq j} -\frac1N N \mu_s^N(j) \int q (\{k\}| j,a,\mu_s^N)\hat\pi^{N,j}_s(da)ds\\ \nonumber
& -\int_0^t \sum_{i\neq j} \frac1N N \mu_s^N(i)\int  q (\{j\}|i,a,\mu_s^N) \hat\pi_s^{N,i}(da)ds
    \\  \label{eq:Martingal}=&
   \mu_t^N(j)-\mu_0^N(j)-\int_0^t \sum_{i\in S} \mu_s^N(i)\int  q (\{j\}|i,a,\mu_s^N) \hat\pi_s^{N,i}(da)ds.
\end{align} 
With this representation we can prove that the sequence of stochastic processes $(M^N(j))$ converges weakly (denoted by $\Rightarrow$) in the Skorokhod $J_1$-topology to the zero process. The proof of this lemma together with the proof of the next theorem can be found in the appendix.

\begin{lemma}\label{lem:martingale} We have for all $j\in S $ that
$$ (M_t^N(j))_{t\ge 0} \Rightarrow 0, \quad N\to\infty.$$
\end{lemma}

Next we show that an arbitrary state-action process sequence is relatively compact which implies the existence of converging subsequences.

\begin{theorem}   \label{theo:subsequence}     
        A sequence of arbitrary state-action  processes $(\mu^N,\hat\pi^N)_N$ is relatively compact. Thus, there exists a subsequence $(N_k)$ which converges weakly 
        $$(\mu^{N_k},\hat\pi^{N_k}) \Rightarrow (\mu,\hat\pi), \mbox{ for } k\to\infty.$$
        Moreover, the limit  $(\mu,\hat\pi)$  satisfies 
        \begin{itemize}
            \item[a)] $(\mu_t) $ has a.s.\ continuous paths,
            \item[b)] and for each component $j$ we have 
            $$\mu_t(j) = \mu_0(j) + \int_0^t \sum_{i\in S} \mu_s(i) \int q(\{j\}|i,a,\mu_s) \hat\pi^{i}_s(da)ds. $$
        \end{itemize}
\end{theorem}

\section{The deterministic limit model}\label{sec:F}
Consider the following deterministic optimization problem:
\begin{align*}
    (F)   & \quad\quad \sup_{\hat\pi} \int_0^\infty e^{-\beta t} r(\mu_t,\hat\pi_t) dt,\\
    & \quad\quad s.t.\ \mu_0\in \PP(S),\; \hat\pi_t^i \in \PP(A),\; \hat\pi_t^i(D(i))=1, \\
    & \quad\quad \hspace*{0.7cm} \mu_t(j) = \mu_0(j) + \int_0^t \sum_{i\in S} \mu_s(i) \int q(\{j\}|i,a,\mu_s) \hat\pi_s^{i}(da)ds, \quad \forall t\ge 0, j=1,\ldots,\vert S\vert.
\end{align*}
Note that the theory of continuous-time Markov processes implies that $\mu_t$ is automatically a distribution. Hence one of the $\vert S\vert$ differential equations in $(F)$ may be skipped.  Also note that when the transition intensity and the reward are linear in the action, relaxation of the control is unnecessary.
We denote the maximal value of this problem by $V^F(\mu_0).$ We show next, that this value provides an asymptotic upper bound to the value of problem \eqref{prob:II}.

\begin{theorem}\label{theo:upperbound}
    For all  $(\mu^N_0) \subset \PP_N(S), \mu_0\in \PP(S)$ with $\mu_0^N \Rightarrow \mu_0$ and for all sequences of policies $(\hat\pi_t^N)$ we have
     $$ \limsup_{N\to\infty } V^N_{\hat\pi^N}(\mu_0^N) \le V^F(\mu_0).$$
\end{theorem}

\begin{proof}
According to Theorem \ref{theo:subsequence}  we can choose a subsequence $(N_k)$ of corresponding state and action processes such that           $$(\mu^{N_k},\hat\pi^{N_k}) \Rightarrow (\mu,\hat\pi), \mbox{ for } k\to\infty.$$
  For convenience we still denote this sequence by $(N).$ We show that 
  \begin{align*}
    \lim_{N\to\infty } V^N_{\hat\pi^N}(\mu_0^N) & = \lim_{N\to\infty } \EE\Big[ \int_0^\infty e^{-\beta t} r(\mu_t^N,\hat\pi_t^N)dt\Big]\\&= \EE\Big[ \int_0^\infty e^{-\beta t} r(\mu_t,\hat\pi_t)dt\Big] \le V^F(\mu_0).
  \end{align*}
The last inequality is true due to the  fact that by Theorem \ref{theo:subsequence} the limit process $ (\mu,\hat\pi)$ satisfies the constraints  of problem $(F)$.

Let us show the second equality. We obtain by bounded convergence ($r$ is bounded)
 \begin{align*}
 & \lim_{N\to\infty } \EE\Big[ \int_0^\infty e^{-\beta t} r(\mu_t^N,\hat\pi_t^N)dt\Big]= \EE\Big[   \int_0^\infty e^{-\beta t} \lim_{N\to\infty }r(\mu_t^N,\hat\pi_t^N)dt\Big]. 
  \end{align*}
Further we have
\begin{align*}
    & \left| \int_0^\infty e^{-\beta t} \sum_{i\in S}\int_A r(i,a,\mu_t^N) \hat\pi_t^{N,i}(da) \mu_t^N(i)dt -  \int_0^\infty e^{-\beta t} \sum_{i\in S}\int_A r(i,a,\mu_t) \hat\pi_t^{i}(da) \mu_t(i)dt\right|\\
    \le & \left| \int_0^\infty e^{-\beta t} \sum_{i\in S}\int_A r(i,a,\mu_t^N) \hat\pi_t^{N,i}(da) \mu_t^N(i)dt -  \int_0^\infty e^{-\beta t} \sum_{i\in S}\int_A r(i,a,\mu_t^{}) \hat\pi_t^{N,i}(da) \mu_t(i)dt\right|\\+&  \left|   \int_0^\infty e^{-\beta t} \sum_{i\in S}\int_A r(i,a,\mu_t^{}) \hat\pi_t^{N,i}(da) \mu_t(i)dt-\int_0^\infty e^{-\beta t} \sum_{i\in S}\int_A r(i,a,\mu_t) \hat\pi_t^{i}(da) \mu_t(i)dt\right|.
\end{align*}
The second expression tends to zero for $N\to\infty$ due to the definition of the Young topology and the fact that $a\mapsto r(i,a,\mu)$ is continuous by (R2). The first expression can be bounded from above by
\begin{align*}
   & \int_0^\infty e^{-\beta t} \sum_{i\in S}\int_A \left| r(i,a,\mu_t^N)  \mu_t^N(i) - r(i,a,\mu_t^{})\mu_t(i) \right| \hat\pi_t^{N,i}(da) dt \\
    \le & \int_0^\infty e^{-\beta t} \sum_{i\in S} \sup_{a\in D(i)} \left| r(i,a,\mu_t^N)  \mu_t^N(i) - r(i,a,\mu_t)\mu_t(i) \right|  dt
\end{align*}
which also tends to zero for $N\to\infty$ due to (R1), (R2), Lemma \ref{lem:unifconf} and dominated convergence. Thus, the statement follows.
\end{proof}

On the other hand we are now able to construct a strategy which is asymptotically optimal in the sense that the upper bound in the previous theorem is attained in the limit. Suppose that $(\mu^*,\hat\pi^*)$ is an optimal state-action trajectory for problem $(F)$. Then we can consider for the $N$ agents problem the strategy
$$ \hat\pi_t^{N,i}:= \hat\pi^{*,i}_t $$
which applies at time $t$ the kernel $\hat\pi^{*,i}_t$ irrespective of the state $\mu_t^N$ the process is in. More precisely, the considered strategy is deterministic and not a feedback policy.

\begin{theorem}\label{theo:asymptotic1}
    Suppose $\hat\pi^*$ is an optimal strategy for $(F)$ where the corresponding differential equation in $(F)$ has a unique solution and let  $(\mu^N_0) \subset \PP_N(S)$ be such that $\mu_0^N \Rightarrow \mu_0\in \PP(S)$. Then if we use  strategy $\hat\pi^*$ for problem \eqref{prob:II} for any $N$ we obtain
     $$ \lim_{N\to\infty } V_{\hat\pi^*}^N(\mu_0^N) =V^F(\mu_0).$$
     Thus, we call  $\hat\pi^*$  asymptotically optimal.
\end{theorem}

\begin{proof}
    First note that  $\hat\pi^*$ is an admissible policy for any $N$. Further let $(\mu_t^N)$ be the corresponding state process when $N$ agents are present. Since the corresponding differential equation in $(F)$ has a unique solution, every subsequence $(N_k)$ is such that 
     $$\mu^{N_k} \Rightarrow \mu^*, \mbox{ for } k\to\infty$$
     holds (Theorem \ref{theo:subsequence}). Using the same arguments as in the last proof we obtain
     \begin{align*}
          & \lim_{N\to\infty } \EE\Big[ \int_0^\infty e^{-\beta t} r(\mu_t^{N},\hat\pi_t^*)dt\Big]= \EE\Big[   \int_0^\infty e^{-\beta t} r(\mu_t^*,\hat\pi_t^*)dt\Big] = V^F(\mu_0). 
     \end{align*}
     Together with the previous theorem, the statement is shown.
\end{proof}


\begin{remark}
\begin{itemize}
    \item[a)] 
    In order to guarantee the unique solvability, it is sufficient to assume Lipschitz continuity for $\mu \mapsto q(\{j\}|i,a,\mu)$. More precisely, instead of (Q4) we have to assume (Q4') which is given below. The proof follows from the Theorem of Picard-Lindelöf. Example \ref{ex:multiple} shows what may happen if the  differential equation for $(\mu_t)$ in $(F)$ has multiple solutions.
    \item[b)] Note that the construction of asymptotically optimal policies which we present here, works in the same way when we consider control problems with finite time horizon. I.e.\ instead of \eqref{prob:II} we consider

\begin{align}
   \label{prob:finitehorizon}
    & \sup_{\hat\pi}  \EE_\mathbf{x}^{\hat\pi}\Big[ \int_0^T e^{-\beta t} r(\mu^{N}_t,\hat\pi_t)dt+g(\mu^{N}_T)\Big]
\end{align}
with possibly a terminal reward $g(\cdot)$ for the final state.
In this case $(F)$ is given with a finite time horizon 
\begin{align}\nonumber
    & \quad\quad \sup_{\hat\pi} \int_0^T e^{-\beta t} r(\mu_t,\hat\pi_t) dt + g(\mu_T)\\ \label{prob:detfinite}
   & \quad\quad s.t.\ \mu_0\in \PP(S),\; \hat\pi_t^i \in \PP(A),\; \hat\pi_t^i(D(i))=1, \\ \nonumber
    & \quad\quad \hspace*{0.7cm} \mu_t(j) = \mu_0(j) + \int_0^t \sum_{i\in S} \mu_s(i) \int q(\{j\}|i,a,\mu_s) \hat\pi_s^{i}(da)ds, \quad \forall t\in[0,T], j=1,\ldots,|S|.
\end{align}
Theorem \ref{theo:asymptotic1} holds accordingly.
\item[c)] General statements about the existence of optimal controls in $(F)$ can only be made under additional assumptions. A classical result is the Theorem of Filipov-Cesari (see \cite{seierstad} Theorem 8 in Chapter II.8 for the finite time horizon problem and Theorem 15 in Chapter III.7 for the infinite horizon problem). It states the existence of an optimal control (for the finite horizon problem) under the following assumptions:
\begin{itemize}
    \item [i)] There exist admissible pairs $(\hat\pi,\mu),$ (for example by assuming Lipschitz continuity like in a))
    \item[ii)] $A$ is closed and bounded (which we assume here)
    \item[iii)] $\mu$ is bounded for all controls (which we have here)
    \item[iv)] For fixed $\mu$ the set $\{(r(\mu,\alpha)+\gamma, f_1(\mu,\alpha)), \gamma \le 0, \alpha \in A\}$  is convex where $f_1$ is the r.h.s. of the differential equation in $(F)$.
\end{itemize}

\item[d)]Suppose we obtain for problem $(F)$ an optimal feedback rule $\hat\pi_t (\cdot)= \hat\pi(\cdot|\mu_t).$ If $\mu\mapsto \hat\pi(\cdot|\mu)$ is continuous, this feedback rule is also asymptotically optimal for problem \eqref{prob:II}. The proof can be done in the same way as before. If the mapping is not continuous, the convergence may not hold (see application 6.3).
\item[e)] Natural extensions of our model that we have not included in the presentation are resource constraints. For example the total sum of fractions of a certain action may be limited, i.e. we restrict the set $\PP(A)^{|S|}$ by requiring that $\sum_{i\in S} \hat\pi_t^i(\{a^0\}|\mu)\le c < |S|$ for a certain action $a^0\in A.$ As long as the constraint yields a compact subset of $\PP(A)^{|S|}$ our analysis also covers this case.
\end{itemize}
\end{remark}

\begin{example}\label{ex:multiple}
    In this example we discuss what may happen if the differential equation for $(\mu_t)$ in $(F)$ has multiple solutions. Suppose the state space is $S=\{1,2\}$ and the system is uncontrolled. State 1 is absorbing, i.e.\ $q(\{1\}|1,\mu)=q(\{2\}|1,\mu)=0$ (since the system is uncontrolled we skip the action from the notation). 
    So agents can only change from state  2 to 1. The intensity of such a  change is 
$$q(\{1\}|2,\mu)= \left\{ \begin{array}{cc}
  \frac{(\mu_t(1))^\frac13}{1-\mu_t(1)},   & \mbox{if } \mu_t(1) \le 0.99  \\
   \frac{0.99^\frac13}{0.01}  &  \mbox{if } \mu_t(1) \ge 0.99. 
\end{array}\right.$$ Intensities are bounded and continuous. Since the two probabilities $\mu_t(1)+\mu_t(2)=1$ we can concentrate on $\mu_t(1).$ The differential equation for $\mu_t(1)$ in $(F)$ is
$$\mu'_t(1)= \mu_t(1)q(\{1\}|1,\mu_t)+(1-\mu_t(1)) q(\{1\}|2,\mu)  = (\mu_t(1))^\frac13$$
as long as $\mu_t(1) \le 0.99.$
If $\mu_0(1)=0$, there are two solutions of this initial value problem: $\mu_t(1)\equiv 0$ and $\mu_t(1)=(\frac23 t)^\frac32$ for $\mu_t(1) \le 0.99.$  Now consider the following sequence $(\mu_0^N):$ For $N$ even we set $\mu_0^N=(0,1)$ (all $N$ agents start in state 2), for $N$ odd we set $\mu_0^N=(1/N,N-1/N)$ (exactly one agent starts in state 1). Obviously $(\mu_0^N)\Rightarrow (0,1).$ However, when we consider the even subsequence we obtain $\mu_t^N(1)\equiv0$ since the intensity to change from 2 to 1 remains 0. The uneven subsequence converges against the second solution $\mu_t(1)=(\frac23 t)^\frac32$ as long as $\mu_t(1)$ is below 0.99. Thus, when we skip the assumption of a unique solution in Theorem \ref{theo:asymptotic1} we only obtain
 $\limsup_{N\to\infty } V_{\hat\pi^*}^N(\mu_0^N) \le V^F(\mu_0)$, see Theorem \ref{theo:upperbound}.
 \begin{figure}[H]
	\centering
	\includegraphics[height=8.8cm]{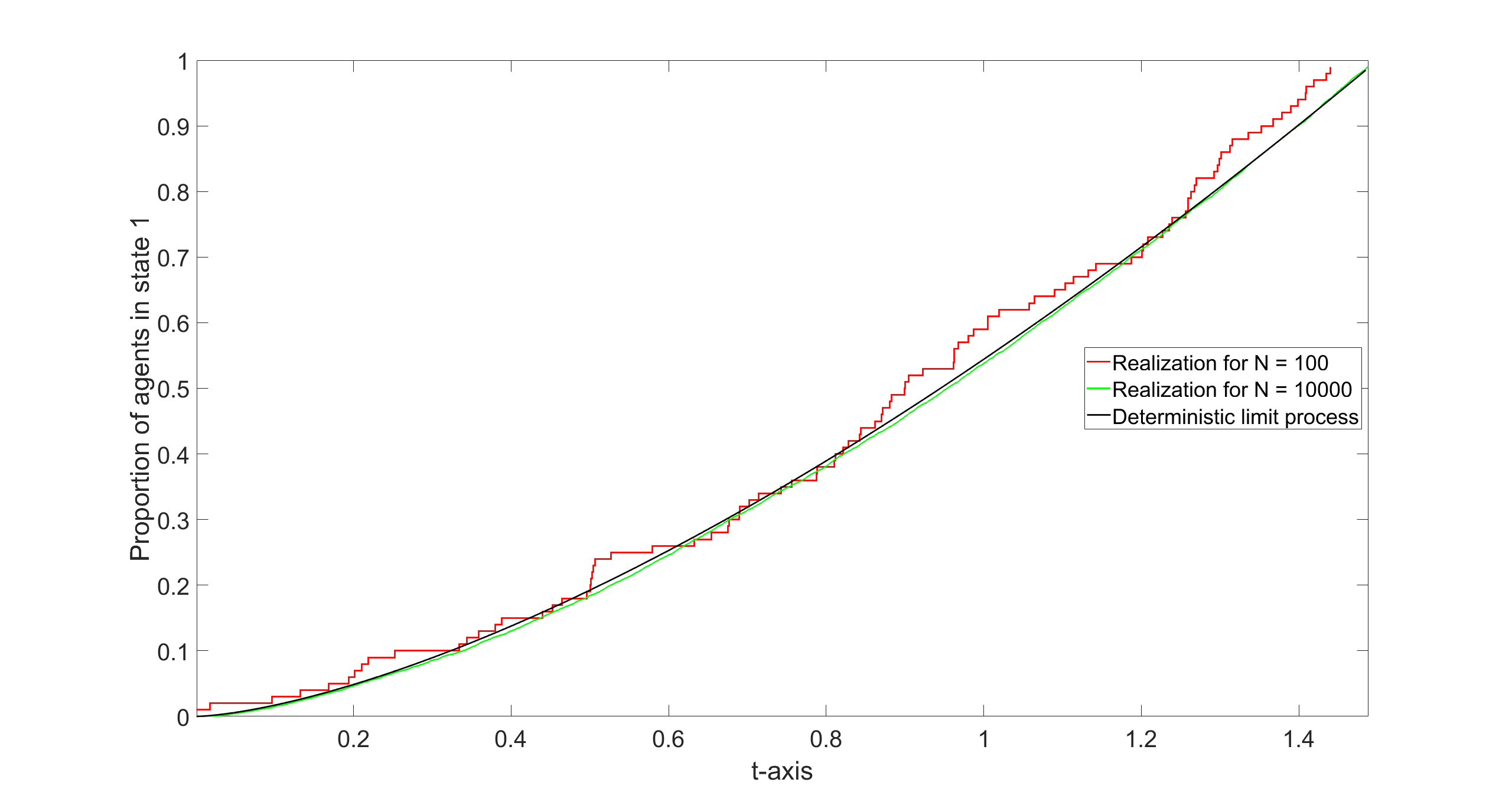}
	\caption{Colourful lines: State trajectories $\mu_t^N(1)$ for $N = 100$ (red) and $N=10000$ (green) agents in Example \ref{ex:multiple} when one agent starts in state 1.\\ Black line: Deterministic limit process $\mu_t(1) = (\frac{2}{3}t)^\frac{3}{2}$.} \label{fig:1}
\end{figure}

\end{example}

Under stricter assumptions it is possible to prove that the rate of convergence in the finite horizon problem \eqref{prob:finitehorizon} is $1/\sqrt{N}$. In order to obtain this rate we need Lipschitz conditions on the reward function and the intensity functions. More precisely assume
\begin{itemize}
    \item[(R1')]  For all $(i,a)\in D$  there exists a uniform constant $L_1>0$  s.t. $$|r(i,a,\mu) -r(i,a,\nu)| \le L_1 \|\mu-\nu\|_{TV}, \quad |g(\mu)-g(\nu)| \le L_1 \|\mu-\nu\|_{TV} $$ for all $\mu,\nu \in \PP(S).$
    \item[(Q4')]  For all $(i,a)\in D, j\in S$  there exists a uniform constant $L_2>0$  s.t. $$|q(\{j\} |i,a,\mu) -q(\{j\}| i,a,\nu)| \le L_2 \|\mu-\nu\|_{TV} $$ for all $\mu,\nu \in \PP(S).$
\end{itemize}
Denote by $\hat\pi^*$ the optimal control of the limiting problem \eqref{prob:detfinite}, $V^{F,T}(\mu_0)$ the corresponding value and let
$$ V_{\hat\pi^*}^{N,T}(\mu_0^N):=  \EE_\mathbf{x}^{\hat\pi^*}\Big[ \int_0^T e^{-\beta t} r(\mu^{N}_t,\hat\pi_t^*)dt+g(\mu^{N}_T)\Big].$$
Then we can state the following convergence rate
\begin{theorem}\label{theo:convrate}
    In the finite horizon setting under assumption (Q1)-(Q5) with (Q4) replaced by (Q4') and (R1'), (R2), suppose that $\E\left[\|\mu^{N}_0-\mu_0\|_{TV}\right] \le \frac{L_0}{\sqrt{N} }$ for a constant $L_0>0$. Then
    $$ \Big| V_{\hat\pi^*}^{N,T}(\mu_0^N) - V^{F,T}(\mu_0)\Big|\le \frac{\tilde L}{\sqrt{N} } $$
    for a constant $\tilde L>0$ which is independent of $N$, but depends on $T.$
\end{theorem}

The statement about the convergence rate can be extended to the infinite horizon problem when the discount factor is large enough. Also note that $\E \left[\|\mu^{N}_0-\mu_0\|_{TV}\right] \le \frac{L_0}{\sqrt{N} }$ is satisfied if e.g.\ the states of the $N$ agents are sampled  i.i.d.\ from $\mu_0.$
\\

A direct implementation of policy $\hat{\pi}^\ast$ in the problem \eqref{prob:II} might make it necessary to update the policy continuously. This can be avoided by using the following policy instead. We assume here that $t\mapsto \hat\pi^*_t$ is piecewise continuous.  Thus, let $(t_n)_{n\in\N}$ be the discontinuity points in time of $\hat{\pi}^\ast$ and define the set $$ \{ T_n^N, n\in\N\} \cup \{  t_n, n\in \N\} =: \{\tilde T_1^N< \tilde T_2^N <\ldots\}$$ where  $T_n^N$ describes the time of the $n$-th jump of the $N$ agents process. Then $(\tilde T_n^N)$ is the ordered sequence of the time points in this set. Define
\begin{equation}\label{eq:policy2}
    \pi_t^{N,\ast} := \sum_{n=0}^\infty \hat\pi_{\tilde T_n}^\ast\mathds{1}_{[\tilde T_n^N, \tilde T_{n+1}^N)}(t).
\end{equation}
The idea of the action process $\pi_t^{N,\ast}$ is to adapt it to $\hat{\pi}^\ast$ only when an agent changes its state or when $\hat{\pi}^\ast$ has a jump, and to keep it constant otherwise.
It can be shown that this sequence of policies is also asymptotically optimal.

\begin{theorem}
    Suppose $\hat\pi^*$ is a   piecewise continuous optimal strategy for $(F)$ where the corresponding differential equation in $(F)$ has a unique solution and let  $(\mu^N_0) \subset \PP_N(S)$ be such that $\mu_0^N \Rightarrow \mu_0\in \PP(S)$. Then if we use the strategy $( \pi_t^{N,\ast})$ of \eqref{eq:policy2} for problem \eqref{prob:II} for any $N$ we obtain
     $$\lim_{N\to\infty } V_{\hat\pi^{N,\ast}}^N(\mu_0^N) =V^F(\mu_0).$$
\end{theorem}

\begin{proof}
    In light of the proof of Theorem \ref{theo:asymptotic1} it is enough to show that $\pi^{N,\ast} \Rightarrow \pi^*.$ Indeed, the convergence can be shown $\PP$-a.s. Now $(\pi^{N,\ast})$ converges in $J_1$-topology against $\pi^*$ on $[0,\infty)$ if and only if $(\pi^{N,\ast})|_{[0,T]}$ the restriction to $[0,T]$ converges in the finite $J_1$-topology to the restriction $\pi^*_{[0,T]}$ for all $T$ which are continuity points of the limit function (see \cite{billingsley2013convergence} sec.16, Lem.1). Since $\hat\pi^*$ is piecewise continuous we can consider the convergence on each compact interval of the partition separately. Indeed we have if $t\in [\tilde T_n^N, \tilde T_{n+1}^N] $
    $$ \vert\vert \pi_t^{N,\ast} - \hat\pi_t^{\ast}\vert\vert_{TV} \le \sup_{s\in [\tilde T_n^N, \tilde T_{n+1}^N]}\vert\vert \hat \pi_s^{\ast} - \hat\pi_t^{\ast}\vert\vert_{TV}. $$
    Since $t\mapsto \hat\pi^*_t$ is continuous on this interval and since all $|\tilde T_{n+1}^N-\tilde T_{n}^N|$ converge to zero for $N\to\infty$ uniformly (the jump intensity increases with $N$) we have that the right hand side converges to zero for $N\to\infty$ uniformly in $t$ which implies the statement.
\end{proof}
\begin{remark}
Let us briefly discuss the main differences to \cite{cecchin2021finite} where a similar model is considered. In   \cite{cecchin2021finite} the author considers a finite horizon problem where model data is not necessarily stationary, i.e. reward and transition intensities may depend on time. Moreover, he solves the corresponding optimization problems ($N$-agents and limit problem) via HJB equations. This requires the notion of viscosity solutions and more regularity assumptions in terms of Lipschitz continuity of reward and transition intensities. Using the MDP perspective,  we can state our solution theorem for the $N$ agents problem (in form of a Bellman equation) and the convergence result under weaker continuity conditions. For the convergence to hold we use randomized policies whereas in \cite{cecchin2021finite} the author sticks to deterministic policies throughout. The obtained convergence rates under Lipschitz assumptions are the same whereas our proof is simpler and more direct. In \cite{cecchin2021finite} the problem is further discussed under stronger assumption. In contrast we present some applications next in order to show how to use the results of the previous sections. 
\end{remark}

\section{Applications}\label{sec:appl}
In this section we discuss two applications of the previously derived theorems and one example which shows that state processes under feedback policies do not necessarily have to converge. More precisely we construct in two applications asymptotically optimal strategies for stochastic $N$ agents systems from the deterministic limit problem $(F)$. The advantage of our problem $(F)$ in contrast to the master equation is that it can be solved with the help of {\em Pontryagin's maximum principle} which gives necessary conditions for an optimal control and is in many cases easier to apply than dynamic programming. For examples see  \cite{avram1995fluid, weiss1996optimal,bauerle2000optimal,bauerle2002optimal} and for the theory see e.g.\ \cite{seierstad,zabczyk2020mathematical}.
\subsection{Machine replacement}
The following application is a simplified version of the deterministic control problem in \cite{thompson1968optimal}. A mean-field application can be found in \cite{huang2016mean}. Suppose a company has $N$ statistically equal machines. Each machine can either be in state 0='working' or in state 1='broken', thus $S=\{0,1\}.$ Two actions are available: 0='do nothing' or 1='repair', thus $A=\{0,1\}$. A working machine does not need repair, so $D(0)=\{0\}.$ The transition rates are as follows: A working machine breaks down with fixed rate $\lambda_{wb}>0$. A broken machine which gets repaired changes to the state 'working' with rate $\lambda_{bw}>0$. Thus, we can  summarize the transition rates of one machine by
\begin{eqnarray*}
    q(\{1\} | 0, 0, \mu_t^N) = \lambda_{wb}, & q(\{0\} | 1, a_t, \mu_t^N) = \lambda_{bw} \delta_{\{a_t=1\}}.
\end{eqnarray*}
The diagonal elements of the intensity matrix are given by
\begin{eqnarray*}
    q(\{0\} | 0, 0, \mu_t^N) = -\lambda_{wb}, & q(\{1\} | 1, a_t, \mu_t^N) = -\lambda_{bw} \delta_{\{a_t=1\}},
\end{eqnarray*}
and all other intensities are zero. Obviously (Q1)-(Q5) are satisfied. The initial state of the system is $\mu_0^N=(1,0)$, i.e. all machines are working in the beginning. Each working machine produces a reward rate $g>0$ whereas we have to pay a fixed cost of $C>0$ when we have to call the service for repair, i.e. $$r(i,a,\mu_t^N)= g \delta_{\{i=0\}}-C \delta_{\{a=1\}}\delta_{\{i=1\}} \frac{1}{1-\mu_t^N(0)}.$$
Hence we obtain an interaction of the agents in the reward. Note that (R1), (R2) are satisfied.
This yields the reward rate for the system
$$ r(\mu_t^N,\hat\pi_t) = g \mu_t^N(0)-C(1-\hat\pi^1_t(\{0\}|\mu_t^N)).$$
Thus, problem $(F)$ in this setting is given by (we denote the limit by $(\mu_t(0),\mu_t(1))=:(\mu_t^0,1-\mu_t^0)$ and let $\alpha_t^0:= \hat\pi_t^1(\{0\}|\mu_t)$):
\begin{align*}
 (F)  & \quad\quad \sup_{(\alpha_t)} \int_0^T  g\cdot \mu_t^0-C\cdot(1-\alpha_t^0)dt,\\
   & \quad\quad s.t.\  \mbox{ for all } t\in[0,T]\\
    & \quad\quad \hspace*{0.7cm} 
   \mu_t^0 = 1 + \int_0^t \lambda_{bw}(1-\mu^0_s)(1-\alpha_s^0) -\lambda_{wb}\mu_s^0ds.
\end{align*}
We briefly explain how to solve this problem using Pontryagin's maximum principle. The Hamiltonian function to $(F)$ is given by
\begin{eqnarray*}
    H(\mu_t^0,\alpha_t^0,p_t,t)&=& g\mu_t^0 -C(1-\alpha_t^0)+p_t (\lambda_{bw}(1-\mu_t^0)(1-\alpha_t^0) -\lambda_{wb}\mu_t^0)\\
    &=& (1-\alpha_t^0)(\lambda_{bw} p_t(1-\mu_t^0)-C) +g\mu_t^0-\lambda_{wb} p_t\mu_t^0
\end{eqnarray*}  
where $(p_t)$ is the adjoint function. Pontryagin's maximum principle yields the following sufficient conditions for optimality (\cite{seierstad,zabczyk2020mathematical}):

\begin{lemma}
    The control $(\alpha_t^{0,*})$ with the associated trajectory $(\mu_t^{0,*})$
is optimal for $(F)$ if there exists a continuous and piecewise continuously differentiable function $(p_t)$ such that for all $t>0$:
\begin{itemize}
    \item[(i)] $\alpha_t^{0,*}$ maximizes $\alpha \mapsto  H(\mu_t^0,\alpha,p_t,t)$ for $\alpha \in[0,1],$
    \item[(ii)] $\dot p_t = -g+p_t(\lambda_{wb}+\lambda_{bw}(1-\alpha_t^0))$ at those points where $p_t$ is differentiable,
    \item[(iii)] $p(T)=0.$
\end{itemize}
\end{lemma}
Inspecting the Hamiltonian it is immediately clear from (i) that the optimal control is essentially 'bang-bang'. For a numerical illustration we solved $(F)$ for the parameters $C=1, g=2, \lambda_{wb}=1, \lambda_{bw}=2$ and $T=4.$ Here it is optimal to do nothing until time point $t^*=\ln{2}.$ Then it is optimal to repair the fraction $\alpha^{0,*}=1/2$ of the broken machines which keeps the number of working machines at $1/2.$ Finally, $\ln{2}$ time units before the end, we do again nothing and wait until the end of the time horizon. A numerical illustration of the optimal trajectory $\mu_t^{0,*}$ of the deterministic problem together with simulated paths under this policy for different number of $N$ can be found in Figure \ref{fig:1}, left. A number of different simulations for $N=1000$ are shown in Figure \ref{fig:1}, right. The simulated paths are quite close to the deterministic trajectory. \\
\begin{figure}[H]
	\centering
	\includegraphics[height=4.2cm]{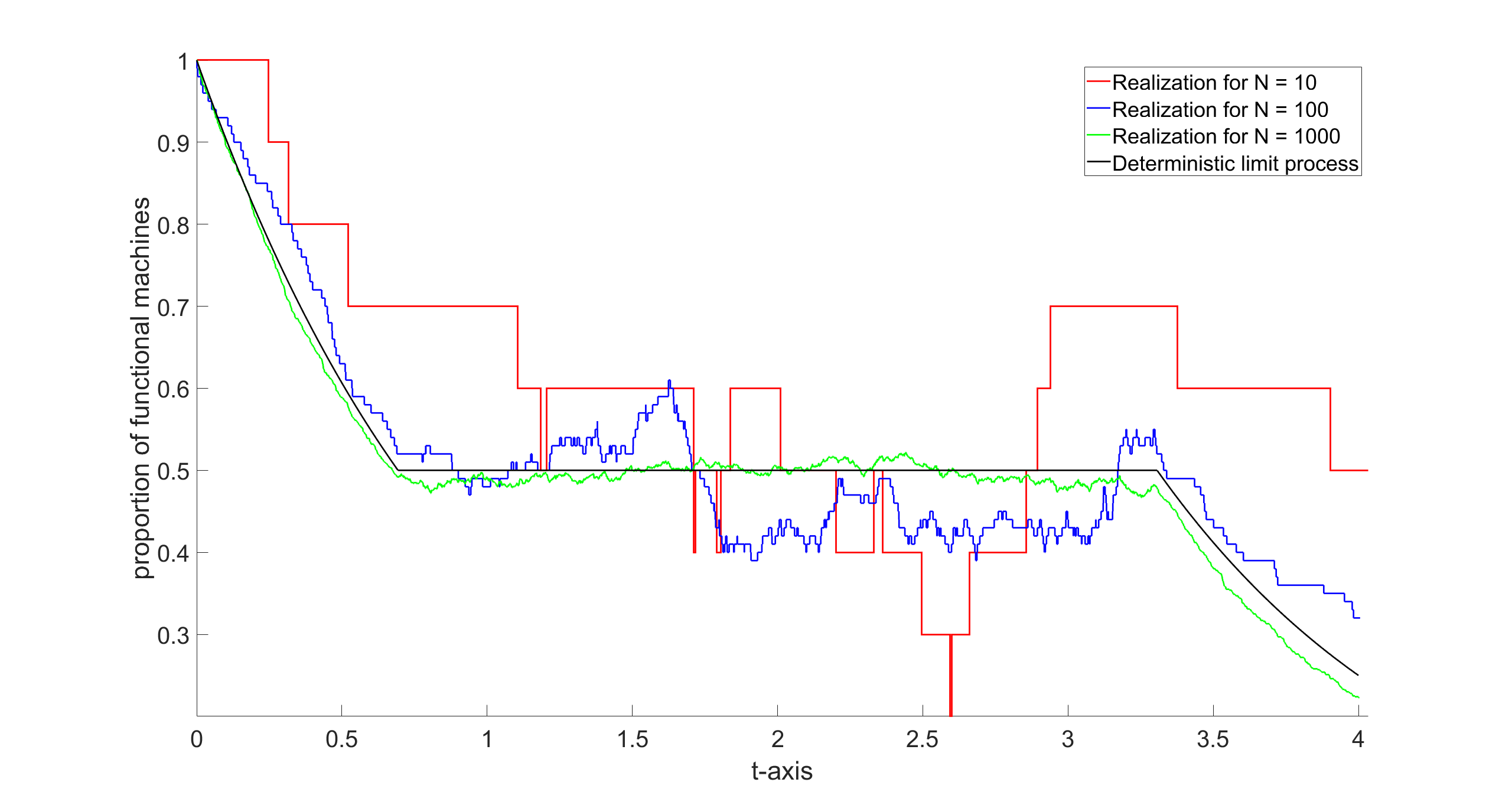}	\includegraphics[height=4.2cm]{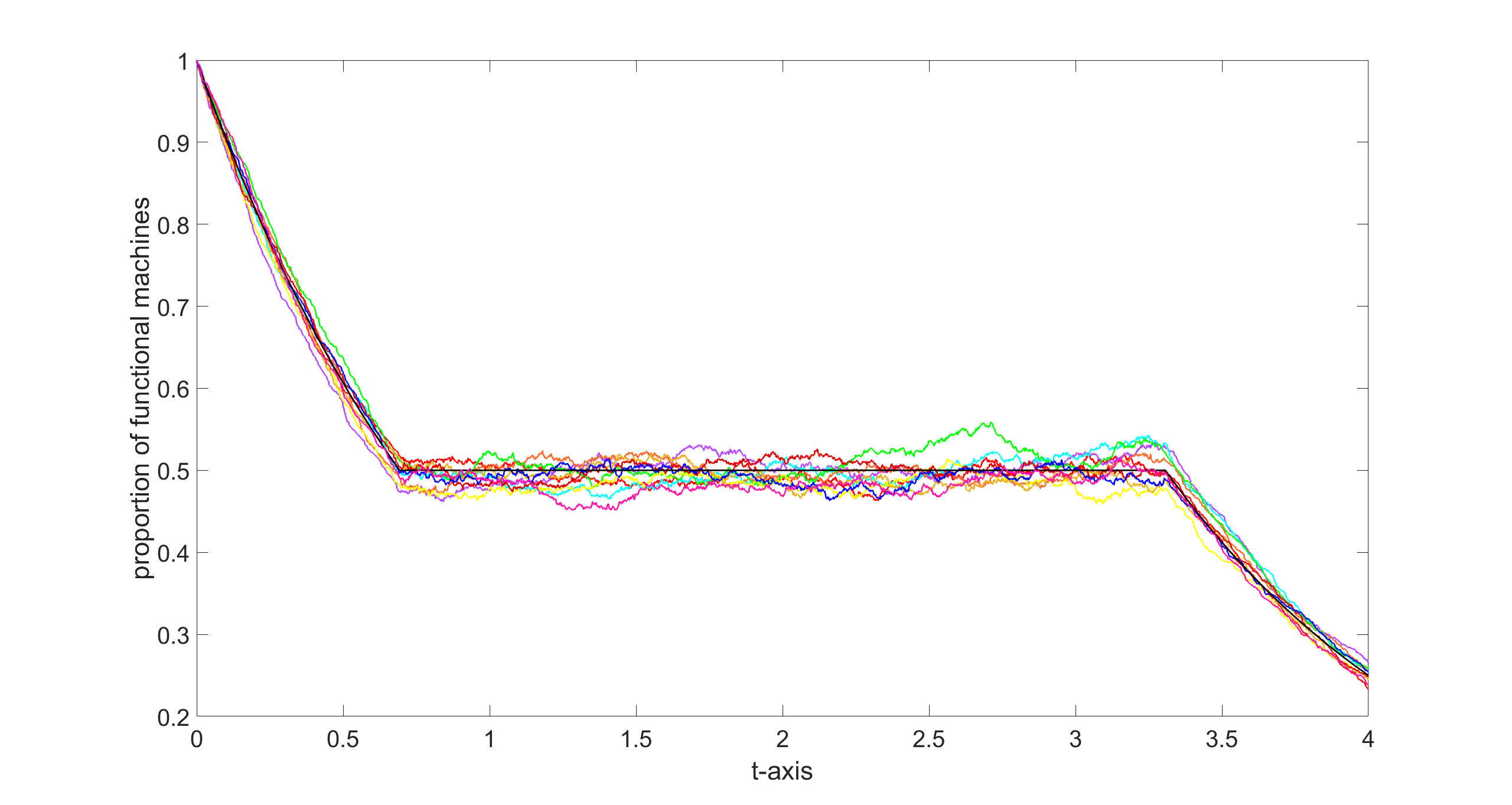}\\
	\caption{Left: State trajectories for different numbers $N$  of machines executing the optimal control for $(F)$. Right: Ten state trajectories for $N = 1000$ machines executing the asymptotically optimal control.} \label{fig:1}
\end{figure}
The optimal  value in the deterministic model is $V^F(1,0) = \frac{9}{2}-\frac{3}{2} \ln(2) \approx 3.4603$. If we simulate ten times the trajectory of the state process for $N = 1000$ machines while following the asymptotically optimal policy and take the average of the respective values, we obtain a mean of $3.43612$ which is slightly less than the value for $(F)$, cp. Theorem \ref{theo:upperbound}.

\subsection{Spreading malware}
This example is based on the deterministic control model considered in 
\cite{khouzani2012maximum}, see also \cite{gast2012mean} and treats the propagation of a virus in a mobile wireless network. It is based on the classical SIR model by Kermack–McKendrick, \cite{daley2001epidemic}. Suppose there are $N$ devices in the network. A device can be in one of the following states: {\em Susceptible (S), Infective (I), Dead (D) } or {\em Recovered (R).} A device is in the susceptible state  if it is not contaminated yet, but prone to infection. A device is infective if it is contaminated by the virus. It is dead if the virus has destroyed the software and recovered if the device has already a security patch  which makes it immune to the virus. The states $D$ and $R$ are absorbing. The joint process $\mu_t^N=(S_t^N,I_t^N,D_t^N,R_t^N)$ is a controlled continuous-time Markov chain where $X_t^N$ represents the fraction of devices in  state $X\in\{S,I,D,R\}$. The control is a strategy of the virus which chooses the rate $a(t)\in [0,\bar a]$, at which infected devices are destroyed. In this model we have $S_t^N+I_t^N+D_t^N+R_t^N=1$ and $S_t^N,I_t^N,D_t^N,R_t^N\ge 0$. The transition rates of one device are as follows: A susceptible device gets infected with rate $\lambda_{SI} I_t$ with $\lambda_{SI} >0.$ The rate is proportional to the number of infected devices and we thus have an interaction of one agent with the empirical distribution of the others.  And it gets recovered with rate $\lambda_{SR}>0$ which is the rate the security patch is distributed.  An infected device gets killed by the virus with rate $a(t)\in [0,\bar a]$ chosen by the attacker and gets recovered at rate $\lambda_{IR}>0.$ The rates are shown in the following figure:

\begin{figure}[ht]
\centering
\begin{tikzpicture}
\node[shape=circle,fill=blue!20, draw=blue!60] (S) at (-3.5,0) {S};
\node[shape=circle,fill=blue!20, draw=blue!60] (I) at (0.5,0) {$I$}; \node[shape=circle,fill=blue!20, draw=blue!60] (R) at (3.5,2) {$R$}; \node[shape=circle,fill=blue!20, draw=blue!60] (D) at (3.5,-1) {$D$};
  
     \path [->] (S) edge node[above] {$\lambda_{SR}$} (R);
    \draw[->] (S) edge node[below] {$\lambda_{SI} I_t^N$} (I);
  \draw[->] (I) edge node[below] {$ \quad \lambda_{IR}$} (R);
    \draw[->] (I) edge node[below] {$a_t$} (D);

\end{tikzpicture}
\caption{Transition intensities of one device between the possible states.}
\end{figure}
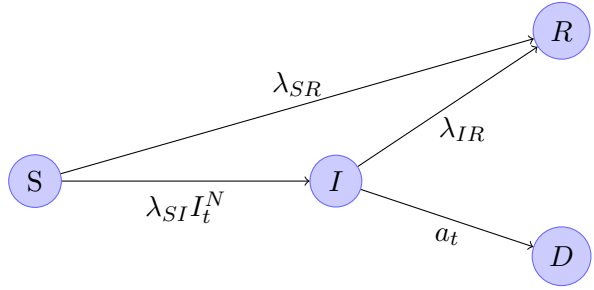

The intensities of one device at time $t$ are summarized by
\begin{eqnarray*}
    q(\{I\} | S, \cdot, \mu_t^N) = \lambda_{SI} I_t^N, & q(\{R\} | S, \cdot, \mu_t^N) = \lambda_{SR},\\
    q(\{D\} | I, a_t, \mu_t^N) = a_t, &    q(\{R\} | I, \cdot, \mu_t^N) = \lambda_{IR}.
\end{eqnarray*}
Thus, the diagonal elements of the intensity matrix are given by
\begin{eqnarray*}
    q(\{S\} | S, \cdot, \mu_t^N) = -\lambda_{SI} I_t^N-\lambda_{SR},\qquad  q(\{I\} | I, a_t, \mu_t^N) = -a_t-\lambda_{IR},  
\end{eqnarray*}
$$q(\{D\} | D, \cdot, \mu_t^N) =     q(\{R\} | R, \cdot, \mu_t^N) = 0$$
and all other intensities are zero. Note that (Q1)-(Q5) are satisfied and that since the intensities are linear in $a$, there is no need for a relaxed control. The initial state of the network is $\mu_0^N=(S_0^N,I_0^N,D_0^N,R_0^N)=(1-I_0,I_0,0,0)$ with $0<I_0<1.$ The aim of the virus is to produce as much damage as possible over the time interval $[0,T]$, evaluated by 
$$ \EE\left[ D_T^N + \frac1T \int_0^T (I^N_t)^2 dt\right]$$
which is given when we choose $r(i,a,\mu)=\frac{1}{T}(\mu(2))^2$ (the second component of $\mu$ squared) and an appropriate terminal reward. (R1) and (R2) are satisfied. Thus, problem $(F)$ in this setting is given by (we denote the limit by $\mu_t=(S_t,I_t,D_t,R_t)$)
\begin{align*}
 (F)   & \quad\quad \sup_{(a_t)} D_T + \frac1T \int_0^T I^2_t dt,\\
   & \quad\quad s.t.\; a_t \in [0,\bar a], \mbox{ and for all } t\in[0,T]\\
    & \quad\quad \hspace*{0.7cm} 
    S_t = 1-I_0 + \int_0^t -\lambda_{SI} I_sS_s-\lambda_{SR} S_sds, \\
    & \quad\quad \hspace*{0.7cm} 
    I_t  = I_0 + \int_0^t \lambda_{SI} I_sS_s-\lambda_{IR} I_s -a_t I_sds, \\
     & \quad\quad \hspace*{0.7cm} 
    D_t  =  \int_0^t a_t I_sds.
\end{align*}
A solution of this deterministic control problem can be found in \cite{khouzani2012maximum}. It is shown there that a critical time point $t_1\in[0,T]$ exists such that $a_t=0$ on $t\in [0,t_1]$ and $a_t=\bar a$ on $t\in(t_1,T].$ Thus, the attacker is not destroying devices from the beginning because this lowers the number of devices which can get infected. Instead, she first waits to get more infected devices before setting the kill rate to a maximum.
\begin{figure}[H]
	\centering
	\includegraphics[height=9cm]{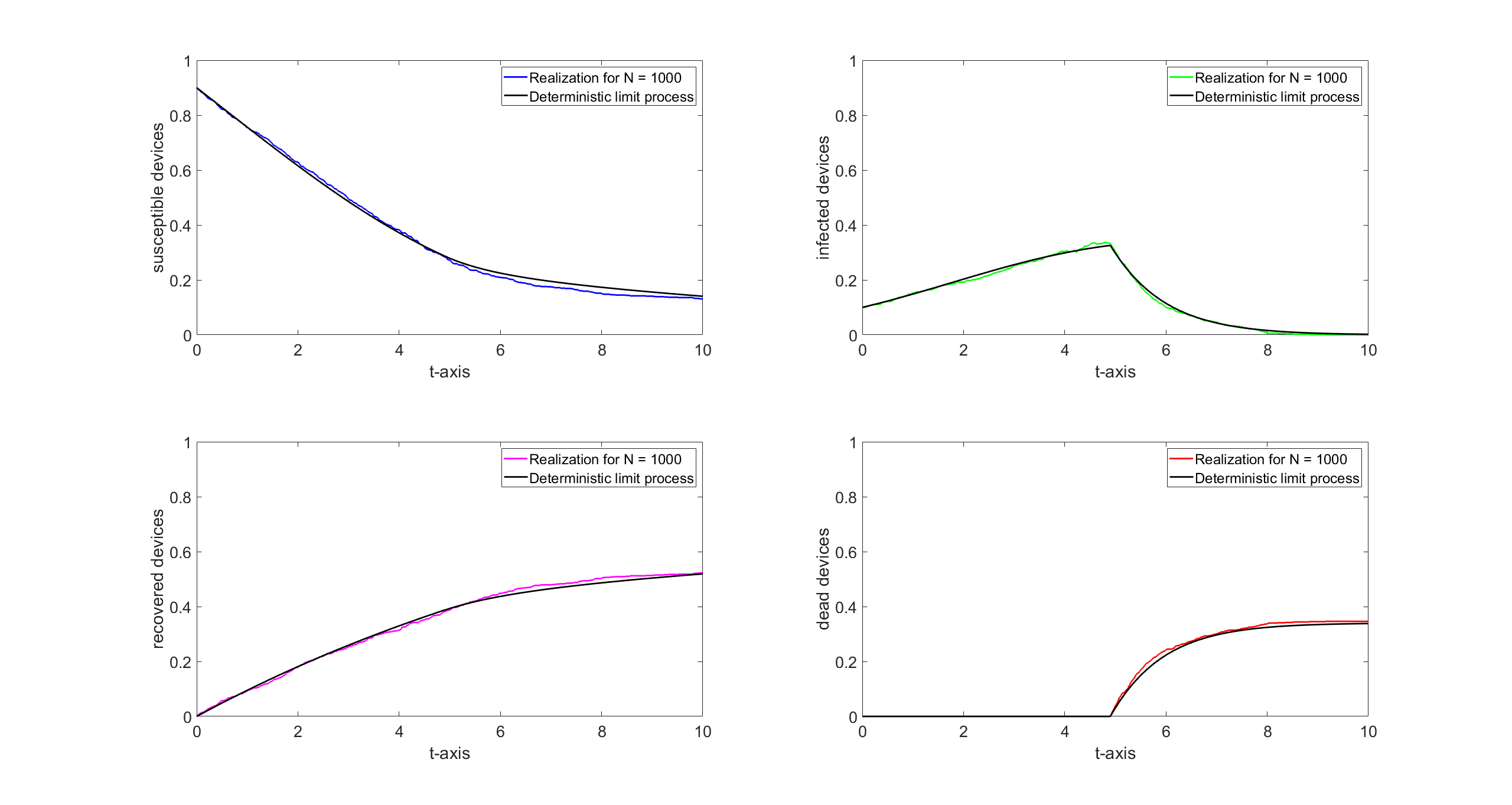}\\
	\caption{State trajectories for $N=1000$ devices under optimal control for $\lambda_{SI}=0.6, \lambda_{SR}=\lambda_{IR}=0.2, \bar{a}=1, T=10.$}\label{fig:3}
\end{figure}
A numerical illustration can be found in Figure \ref{fig:3}. There we can see the trajectories of the optimal state distribution in $(F)$ and simulated paths for $N=1000$ devices for $\lambda_{SI}=0.6, \lambda_{SR}=\lambda_{IR}=0.2, \bar{a}=1, T=10.$ The optimal time point for setting $a_t$ to the maximum is here 4.9. The simulated paths are almost indistinguishable from the deterministic trajectories.

\subsection{Resource competition}
This example shows that  feedback policies in the deterministic problem are not necessarily asymptotically optimal when implemented in the $N$ agents problem. The infinite horizon problem $(F)$ could also be solved using an HJB equation which would provide (under sufficient regularity) a feedback control $\hat\pi(\cdot|\mu)$. I.e. we obtain the optimal control by $\hat\pi_t = \hat\pi(\cdot|\mu_t)$. This feedback function could also be used in the $N$ agents model. However, in this case convergence of the $N$ agents model to the deterministic model like in Theorem \ref{theo:asymptotic1} is not guaranteed. Convergence may fail when discontinuities in the feedback function are present. The example is an adaption of the queuing network considered in \cite{kumar1989dynamic,rybko1992ergodicity} to our setting. Suppose the state space is given by $S=\{1,2,3,4,5,6,7,8\}.$ Agents starting in state 1 change to state 2, then 3 and are finally absorbed in state 4. Agents starting in state 5 change to state 6, then 7 and are finally absorbed in state 8. The aim is to get the agents in the absorbing states as quickly as possible by activating the intensities in states 2,3,6 and 7.  The intensity for leaving states $1$ and $5$ is $\lambda_1=\lambda_5=1$, the full intensity for leaving states $2$ and $6$ is $\lambda_2=\lambda_6=6$ and  finally the full intensity for leaving states $3$ and $7$ is $\lambda_3=\lambda_7=1.5.$ The action space is $A=\{0,1\}$ where actions have to be taken in states $2,3,6$ and $7$ and determine the activation of the transition intensity. Action $a=0$ means that the intensity is deactivated and $a=1$ that it is fully activated.  There is a resource constraint such that the sum of the activation probabilities in states $2$ and $7$ as well as the sum of the activation probabilities in states $3$ and $6$ are constraint by 1 (see remark on p.13). When we denote the randomized control by $\hat\pi_t^2= a_t, \hat\pi_t^7= 1-a_t, \hat\pi_t^6= b_t, \hat\pi_t^3= 1-b_t$, $a_t,b_t\in[0,1]$  then
the intensities are given by
\begin{eqnarray*}
    q(\{3\} | 2,a_t, \mu_t^N) = a_t \lambda_2,&&  q(\{4\} | 3, 1-b_t, \mu_t^N) = (1-b_t) \lambda_3,\\
     q(\{7\} | 6,b_t, \mu_t^N) = b_t\lambda_6, && q(\{8\} | 7, 1-a_t, \mu_t^N) = (1-a_t)\lambda_7.
\end{eqnarray*}
An illustration of this model can be seen in Figure \ref{fig:4}.

\begin{figure}[ht]
\centering
\begin{tikzpicture}
\node[shape=circle,fill=blue!20, draw=blue!60] (1) at (-4.5,1) {1};
\node[shape=circle,fill=blue!20, draw=blue!60] (2) at (-1.5,1) {2}; \node[shape=circle,fill=blue!20, draw=blue!60] (3) at (1.5,1) {3}; \node[shape=circle,fill=blue!20, draw=blue!60] (4) at (4.5,1) {4};

\node[shape=circle,fill=blue!20, draw=blue!60] (8) at (-4.5,-1) {8};
\node[shape=circle,fill=blue!20, draw=blue!60] (7) at (-1.5,-1) {7}; \node[shape=circle,fill=blue!20, draw=blue!60] (6) at (1.5,-1) {6}; \node[shape=circle,fill=blue!20, draw=blue!60] (5) at (4.5,-1) {5};

\draw[blue, very thick] (-2,-1.5) rectangle (-1,1.5);
\draw[blue, very thick] (1,-1.5) rectangle (2,1.5);
  
     \path [->] (1) edge node[above] {$\lambda_1$} (2);
   \draw[->] (2) edge node[above] {$a_t \lambda_2$} (3);
    \draw[->] (3) edge node[above] {$(1-b_t) \lambda_3$} (4);
  
    \draw[->] (5) edge node[below] {$\lambda_5$} (6);
    \draw[->] (6) edge node[below] {$b_t\lambda_6$} (7);
    \draw[->] (7) edge node[below] {$(1-a_t)\lambda_7$} (8);

\end{tikzpicture}
\caption{Transition intensities of one agent for the resource constraint problem.}\label{fig:4}
\end{figure}
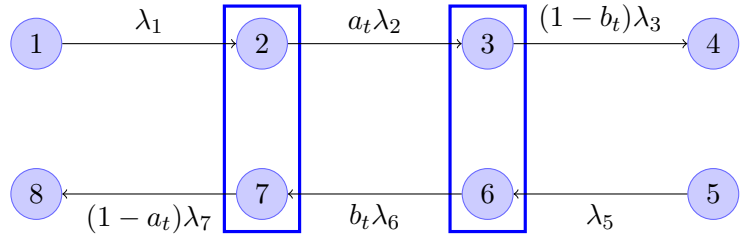
The initial state distribution is given by $\mu_0=(\frac{5}{14},\frac{1}{14},\frac{1}{14},0,\frac{5}{14},\frac{1}{14},\frac{1}{14},0)$ where we assume for the simulation that we have $N=1400$ agents.
Now suppose further that agents in the absorbing states 4 and 8 produce no cost whereas agents in state 3 and 7 are the most expensive as soon as there are at least $0.01\%$ of the population present. This optimization criterion leads to a priority rule where agents in state 3 receive priority (and thus full capacity) over those in state 6 (as long as there are at least $0.01\%$ present) and agents in state 7 receive priority (and thus full capacity) over those in state 2 (as long as there are at least $0.01\%$ present). In the deterministic problem the priority rule can be implemented such that once the number of agents in state 3 and 7 fall to the threshold  of $0.01\%$ of the population it is possible to keep this level. This is not possible in the $N$ agents problem. The priority switch leads to blocking the agents in the other line, see Figure \ref{fig:5}. The blue line shows the state trajectories in the deterministic model. The red line is a realization of the system for $N=1400$ agents where we use the deterministic open-loop control of Theorem \ref{theo:asymptotic1}. We see that the state processes converge. Finally the green line is a realization of the $N=1400$ agents model under the priority rule. We can see that here state processes do not converge.
\begin{figure}[H]
	\centering
	\includegraphics[height=9cm]{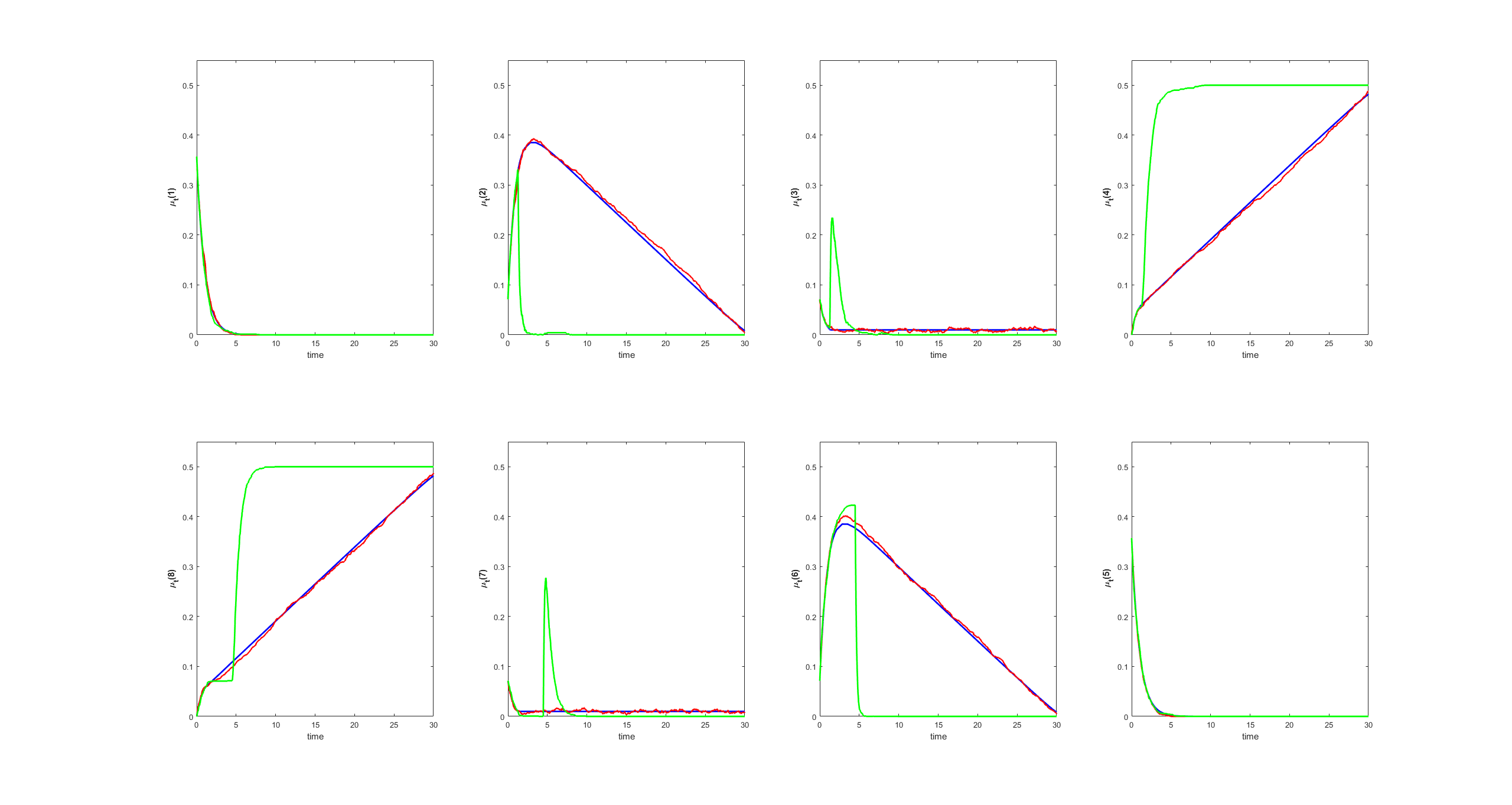}\\
	\caption{State trajectories for $N=1400$ agents. Deterministic trajectory (blue), realization under deterministic open loop (red),  realization under feedback priority rule (green). }\label{fig:5}
\end{figure}

\section{Appendix}
\subsection{Auxiliary result}
\begin{lemma}\label{lem:unifconf}
Let $X$ be a separable metric space, $Y$ be compact metric and $f:X \times Y\to \R$ continuous. Then $x_n\to x$ for $n\to\infty$ implies
$$\lim_{n\to \infty} \sup_{y\in Y} |f(x_n,y)-f(x,y)|=0. $$ 
\end{lemma}
For a proof see e.g. Lemma B.12, \cite{lange2017cost}.

\subsection{ Proof of Theorem \ref{theo:equivalence} }
First of all observe that the reward function $r$ in \eqref{eq:reward1} in the $N$ agents problem is symmetric, i.e. $r(\mathbf{x},\mathbf{a})=r(s(\mathbf{x}),s(\mathbf{a}))$ for any permutation $s(\cdot)$ of the vectors. Moreover, the agent transition intensities $q(\cdot|i,a,\mu[\mathbf{X}_t])$ depend only on the own state of the agent and on $\mu[\mathbf{X}_t].$ Thus, the optimal policy in the $N$ agents problem at time $t$ only depends on $\mu[\mathbf{X}_t]$. Now for a decision rule $\pi$ for the $N$ agents problem define for all states $i\in S:$
     $$ \hat\pi^i(da|\mu) := \frac{1}{N\mu(i)} \sum_{k=1}^N \pi^k(da|\mathbf{x}) \mathds{1}_{\{x^k=i\}}$$
     where $\mu=\mu[\mathbf{x}].$
     On the right-hand side we consider all agents in state $i$ and take a convex combination of their action distributions as the action distribution in state $i.$
     If $\pi$ depends only on $\mu[\mathbf{x}],$ then this is also true for $\hat\pi.$

     Choosing $\hat\pi$ in the measure-valued MDP yields the reward (again $\mu=\mu[\mathbf{x}]$)
     \begin{align*}
         &r(\mu,\hat\pi)= \sum_{i\in S} \int r(i,a,\mu) \frac{1}{N\mu(i)} \sum_{k=1}^N \pi^k(da|\mathbf{x}) \mathds{1}_{\{x^k=i\}} \mu(i)\\=& \frac1N \sum_{k=1}^N \sum_{i\in S} \mathds{1}_{\{x^k=i\}}  \int r(i,a,\mu)\pi^k(da|\mathbf{x})=r(\mathbf{x},\pi).
     \end{align*}
     Thus, the reward in both formulations is the same. 
     Finally the transition intensity in the $N$ agents model that one agent changes its state from $i$ to $j$ is given by (again $\mu=\mu[\mathbf{x}]$)
     \begin{align*}
     & \sum_{k=1}^N \mathds{1}_{\{x^k=i\}} \int q(\{j\}| i,a,\mu) \pi^k(da|\mathbf{x} )\\
         = & N\mu(i) \int_A q(\{j\}| i,a,\mu)  \frac{1}{N\mu(i)} \sum_{k=1}^N \pi^k(da|\mathbf{x}) \mathds{1}_{\{x^k=i\}}\\= & 
          N\mu(i) \int_A q(\{j\}| i,a,\mu) \hat \pi^i(da|\mu) = q(\{\mu^{i\to j}\}| \mu,{\hat\pi }).
     \end{align*}
      Thus, the empirical measure process of the $N$ agents problem is statistically equal to the measure-valued MDP process and they produce the same expected reward under measure-dependent policies which implies the result. A formal proof has to be done by induction like in \cite{bauerle2023mean} Thm. 3.3.     

\subsection{ Proof of Lemma \ref{lem:martingale} }
First we show that $M_t^N(j)$ is bounded for fixed $t$:
\begin{align*}
   | M_t^N(j)| &= \Big|\mu_t^N(j)-\mu_0^N(j)-\int_0^t \sum_{i\in S} \mu_s^N(i)\int  q (\{j\}|i,a,\mu_s^N) \hat\pi_s^{N,i}(da)ds\Big| \\
    &\leq \vert {\mu_t^N(j)}-{\mu_0^N(j)}\vert+\int_0^t \sum_{i\in S} \mu_s^N(i)\int  {\vert q (\{j\}|i,a,\mu_s^N)\vert} \hat\pi_s^{N,i}(da)ds \\
    &\leq 1+q_{max}\cdot t <\infty
\end{align*}
Therefore $(M_t^N(j))_{t\ge 0}$ are square-integrable martingales. Now we take advantage of the fact that there are only jumps of height $\frac{1}{N}$ in our model, since no two agents change their state simultaneously. With the quadratic variation of the process we obtain
\begin{align}
    \nonumber
    \mathbb{E}[(M_t^N(j))^2] &= \mathbb{E}[\langle M_t^N(j)\rangle] \leq \frac{1}{N^2}\mathbb{E}[ \text{$\#$ jumps in } [0,t]] \\\nonumber
    &\leq \frac{1}{N^2}\cdot N\cdot q_{max}\cdot t = \frac{1}{N}\cdot q_{max} \cdot t \overset{N\to \infty}{\longrightarrow} 0.
\end{align}
Doob's $L^p$-inequality provides on $[0,t]$ 
\begin{equation}
    \nonumber
    \mathbb{E}[(\sup_{s\in [0,t]}M_s^N(j))^2] \leq 4\cdot \mathbb{E}[(M_t^N(j))^2] \overset{N\to \infty}{\longrightarrow} 0.
\end{equation}
Thus for the sequence $(\sup_{s\in [0,t]} M_s^N(j))_{N\in\N}$ it holds that
\begin{equation}
    \nonumber
    \sup_{s\in [0,t]} M_s^N(j) \quad\overset{L^2}{\longrightarrow} \quad0.
\end{equation}
Now we can find a suitable probability space $(\Omega,\mathcal{F},\PP)$, such that for $\PP$-almost all $\omega \in \Omega$ the sequence of functions $((M_s^N(j)(\omega))_{s\in [0,t]})_{N\in \N}$ converges uniformly to the zero-function.\\
The finite-dimensional distributions with arbitrary time-points $t_1,..,t_k\in [0,t]$ then obviously fulfill
\begin{equation}
\nonumber
    (M_{t_1}^N(j),...,M_{t_n}^N(j)) \quad\overset{\text{a.s.}}{\longrightarrow}\quad (0,...,0)
\end{equation}
and therefore in particular
\begin{equation}
\nonumber
    (M_{t_1}^N(j),...,M_{t_n}^N(j)) \quad\Rightarrow\quad (0,...,0).
\end{equation}
Here $\Rightarrow$ is the usual weak convergence of random vectors in $\R^n.$
To apply Theorem VI.16 in \cite{pollard} we check Aldous' condition. Let $(\delta_N)$ be a sequence of positive numbers with $\delta_N \to 0$ and $(\sigma_N)$ a sequence of stopping times w.r.t.\ $(\mathcal{F}_t^N)$ with values in $[0,t]$. Then we have
\begin{align}
    \nonumber
   \EE[(M_{\sigma_N}^N(j))^2] \leq \mathbb{E}[(\sup_{s\in [0,t]}M_s^N(j))^2] \leq 4\cdot \mathbb{E}[(M_t^N(j))^2] \overset{N\to \infty}{\longrightarrow} 0.
\end{align}
Further, for $N$ sufficiently large it holds that
\begin{equation}
    \nonumber
   \EE[(M_{\sigma_N+\delta_N}^N(j))^2] \leq \mathbb{E}[(\sup_{s\in [0,2t]}M_s^N(j))^2] \leq 4\cdot \mathbb{E}[(M_{2t}^N(j))^2] \overset{N\to \infty}{\longrightarrow} 0.
\end{equation}
Therefore $M_{\sigma_N}^N(j)$ and $M_{\sigma_N+\delta_N}^N(j)$ converge in $L^2$ to $0$ (and thus their difference). Hence, the conditions of Theorem VI.16 in \cite{pollard} are fulfilled, and the sequence $M_{t}^N(j)$ converges weakly on $[0,\infty)$ towards $0$ in the sense of the Skorokhod $J_1$-metric.

\subsection{Proof of Theorem \ref{theo:subsequence}}
We start by showing the relative compactness of a sequence $(\mu^N)_N$.
    We use Theorem 2.7 in \cite{kurtz}.
    The sequence $(\mu^N)_{N}$ has paths in $D_{\PP(S)}[0,\infty)$, where $\PP(S)$ is complete and separable with respect to the total variation distance.\\
    In what follows let $\sigma$ be an arbitrary $(\mathcal{F}_t^N)$-stopping time with $\sigma\leq T$  a.s.\\
    For every $\varepsilon>0$ and rational $t\geq0$ choose the compact set $\Gamma_{t,\varepsilon} \equiv \PP(S)$. Then  we obtain by construction of the model
    \begin{equation}
        \nonumber
        \PP(\mu_t^N \in \Gamma_{t,\varepsilon}) = 1.
    \end{equation}
    Moreover, for every $T>0$ it holds that
    \begin{align}
        \nonumber
        &\lim_{\delta \to 0} \limsup_{N\to \infty} \sup_{\sigma} \EE[\min\{1,\vert\vert \mu_\sigma^N - \mu_{\sigma+\delta}^N\vert\vert_{TV}\}] \\\nonumber
        \leq &\lim_{\delta \to 0} \limsup_{N\to \infty} \sup_{\sigma} \EE[\vert\vert \mu_\sigma^N - \mu_{\sigma+\delta}^N\vert\vert_{TV}] \\\nonumber
        \leq & \lim_{\delta \to 0} \limsup_{N\to \infty} \sup_{\sigma} \EE[\# \text{ state changes in }[\sigma,\sigma+\delta]  ] \cdot \frac{1}{N} \\\nonumber
        \leq &\lim_{\delta \to 0} \limsup_{N\to \infty} \sup_{\sigma} N\cdot q_{max}\cdot \delta \cdot \frac{1}{N} =0. 
    \end{align}
    The second inequality holds because $\vert\vert \mu_s^N-\mu_t^N\vert\vert_{TV} = \frac{1}{N}$, provided that in $[s,t]$ only one state change occurs, i.e. one agent changes its state. Theorem 2.7 in \cite{kurtz} now states that $(\mu^N)_{N}$ is relatively compact.\\\\
    Since $\mathcal{R}$ is compact, so is $\mathcal{R}^{\vert S\vert}$ and we obtain directly the relative compactness of $(\hat{\pi}^N)_N$. The relative compactness of the sequence of state-action-processes $(\mu^N,\hat\pi^N)_N$ then follows by Proposition 3.2.4 in \cite{ethierkurtz}. Thus, a converging subsequence exists. To ease the notation we will still denote it by $(N).$
    \\\\
    To prove the continuity of the limit state process define for arbitrary $\mu\in D_{\PP(S)}[0,\infty)$ 
    \begin{align}
    \nonumber
    J(\mu,u) &= \sup_{0\leq t\leq u} \vert\vert \mu_t-\mu_{t-}\vert\vert_{TV}.\\\nonumber
    J(\mu) &= \int_0^\infty e^{-u}J(\mu,u)du.
    \end{align}
    For the sequence of state processes $(\mu^N)_{N}$ we get
\begin{align}
    \nonumber
    \lim_{N\to\infty} J(\mu^N) = \lim_{N\to\infty} \int_0^\infty e^{-u} \sup_{0\leq t\leq u} {\vert\vert \mu_t^N-\mu^N_{t-}\vert\vert_{TV}} du \leq \lim_{N\to \infty} \frac{1}{N} =0.
\end{align}
We exploit the fact that there can be at most jumps of height $\frac{1}{N}$ in the state processes with $N$ agents. Theorem 3.10.2 a) in \cite{ethierkurtz} then implies the a.s.\ continuity of the limit state process $(\mu_t^\ast)_{t\geq0}$.\\\\
In particular, due to the Skorokhod representation theorem we find a probability space such that convergence of $\mu^N\Rightarrow \mu^*$ holds almost surely in $J_1$ and is uniformly on compact sets such as $[0,t]$ since $\mu^*$ is a.s.\ continuous (see p. 383 in \cite{whitt2002stochastic}). Thus, component-wise for almost all $\omega$ in the probability space above we obtain:
\begin{equation}
    \nonumber\lim_{N\to \infty} \ \sup_{0\leq s\leq t} \vert\vert \mu^{N}_s(\omega)-\mu^\ast_s(\omega) \vert\vert_{TV}= 0
\end{equation} for every $t\in [0,\infty)$.
\\\\
    Finally we have to take the limit $N\to\infty$ in \eqref{eq:Martingal}. By the previous Lemma \ref{lem:martingale} we know that the martingale on the left-hand side converges to zero and that $\mu^{N}_t(\omega) \to \mu_t^*$. Now consider the integral on the right-hand side:
    \begin{align*}
      & \left| \int_0^t\! \sum_{i\in S} \mu_s^N(i)\! \int\!\! q(\{j\}|i,a,\mu^N_s) \hat\pi^{N,i}_s(da)ds - \int_0^t\! \sum_{i\in S} \mu_s^*(i)\! \int\!\! q(\{j\}|i,a,\mu^*_s) \hat\pi^{*,i}_s(da)ds\right|\\ 
\le &  \left| \int_0^t\! \sum_{i\in S} \mu_s^N(i)\! \int\!\! q(\{j\}|i,a,\mu^N_s) \hat\pi^{N,i}_s(da)ds - \int_0^t\! \sum_{i\in S} \mu_s^*(i)\! \int\!\! q(\{j\}|i,a,\mu^*_s) \hat\pi^{N,i}_s(da)ds\right|\\ 
+ & \left| \int_0^t\! \sum_{i\in S} \mu_s^*(i)\! \int\!\! q(\{j\}|i,a,\mu^*_s) \hat\pi^{N,i}_s(da)ds - \int_0^t\! \sum_{i\in S} \mu_s^*(i)\! \int\!\! q(\{j\}|i,a,\mu^*_s) \hat\pi^{*,i}_s(da)ds\right|.
\end{align*}
The second expression tends to $0$ for $N\to\infty$ due to the definition of the Young topology and the fact that $a\mapsto q(\{j\}|i,a,\mu^*_s)$ is continuous by assumption. The first expression can be bounded by
\begin{align*}
    & \int_0^t\! \sum_{i\in S} \! \int\!\! \left|  \mu_s^N(i) q(\{j\}|i,a,\mu^N_s)  -  \mu_s^*(i) q(\{j\}|i,a,\mu^*_s)\right| \hat\pi^{N,i}_s(da)ds\\
    \le & \int_0^t\! \sum_{i\in S} \sup_{a\in D(i)} \left|  \mu_s^N(i) q(\{j\}|i,a,\mu^N_s)  -  \mu_s^*(i) q(\{j\}|i,a,\mu^*_s)\right| ds
\end{align*}
which also tends to zero due to dominated convergence, (Q4),(Q5) and Lemma \ref{lem:unifconf}.\\\\
Now  putting things together, equation \eqref{eq:Martingal} implies that the limit satisfies the stated differential equation.

\subsection{Proof of Theorem \ref{theo:convrate}}
    Fix $\hat\pi^*$ and let $(\mu_t)$ be the unique solution of 
    $$ \mu_t(j) = \mu_0(j) + \int_0^t \sum_{i\in S} \mu_s(i) \int q(\{j\}|i,a,\mu_s) \hat\pi_s^{*,i}(da)ds, \quad \forall t\in[0,T], j=1,\ldots,|S|.$$
    Further recall from \eqref{eq:Martingal}
\begin{align}\nonumber
    \mu_t^N(j) =&
   \mu_0^N(j)+\int_0^t \sum_{i\in S} \mu_s^N(i)\int  q (\{j\}|i,a,\mu_s^N) \hat\pi_s^{*,i}(da)ds  + M_t^N(j), \quad \forall t\in[0,T], j=1,\ldots,|S|.
\end{align} 
    Now we obtain
    \begin{align*}
         \Big| V_{\pi^*}^{N,T}(\mu_0^N) - V^{F,T}(\mu_0)\Big| & \le 
          \EE_\mathbf{x}^{\hat\pi^*}\Big[ \int_0^T e^{-\beta t} | r(\mu^{N}_t,\hat\pi_t^*)- r(\mu_t,\hat\pi_t^*)|dt+|g(\mu^{N}_T)-g(\mu_T)|\Big]\\
&\le  \EE_\mathbf{x}^{\hat\pi^*}\Big[ \int_0^T e^{-\beta t} L_3 \|\mu^{N}_t-\mu_t\|_{TV}dt+ L_1 \|\mu^{N}_T-\mu_T\|_{TV}\Big].
\end{align*}
Note that (R1') implies the Lipschitz continuity of $\mu \mapsto r(\mu,\pi)$ with a constant $L_3>0$ since $r$ is bounded. Next consider
\begin{align*}
    |\mu^{N}_t(j)-\mu_t(j)|  \le & |\mu^{N}_0(j)-\mu_0(j)| + \int_0^t   \sum_{i\in S} \mu_s^N(i)\int  | q (\{j\}|i,a,\mu_s^N)-q (\{j\}|i,a,{\mu}_s)| \hat\pi_s^{*,i}(da)ds  \\
& +  \int_0^t   \sum_{i\in S} |\mu_s^N(i)-\mu_s(i)| \int
q (\{j\}|i,a,{\mu}_s) \hat\pi_s^{*,i}(da)|ds + |M_t^N(j)|\\
\le & |\mu^{N}_0(j)-\mu_0(j)| + (L_2+2 q_{max} ) \int_0^t    \|\mu^{N}_s-\mu_s\|_{TV}ds+|M_t^N(j)|.
\end{align*}
For the last term we have already shown in Lemma \ref{lem:martingale} that
$$  \EE_\mathbf{x}^{\hat\pi^*}[ |M_t^N(j)|] \le  \sqrt{\EE_\mathbf{x}^{\hat\pi^*} [(M_t^N(j))^2]} \le \frac{\sqrt{q_{max} t}}{\sqrt{N}}.
 $$
Thus, from the two previous inequalities we get with $L_4 := |S|\sqrt{q_{max}T}/2$ that
\begin{align*}
   \EE_\mathbf{x}^{\hat\pi^*}\big[   \|\mu^{N}_t-\mu_t\|_{TV}\big] &\le  \EE_\mathbf{x}^{\hat\pi^*}\big[   \|\mu^{N}_0  -\mu_0\|_{TV}\big] + \frac{{L_4}}{\sqrt{N}}+ L_5  \int_0^t     \EE_\mathbf{x}^{\hat\pi^*}\big[ \|\mu^{N}_s-\mu_s\|_{TV}\big]ds \\
&\le  \frac{{L_6}}{\sqrt{N}}+ L_5  \int_0^t     \EE_\mathbf{x}^{\hat\pi^*}\big[ \|\mu^{N}_s-\mu_s\|_{TV}\big]ds.
\end{align*}
Finally, Gronwall's inequality implies that for all $t\in [0,T]$
$$  \EE_\mathbf{x}^{\hat\pi^*}\big[   \|\mu^{N}_t-\mu_t\|_{TV}\big] \le  \frac{{L_7}}{\sqrt{N}} e^{L_5 T}$$
which in turn implies the statement.

\bibliographystyle{apalike}
\bibliography{literatur}

\section{Statements and Declarations}
The authors have no relevant financial or non-financial interests to disclose.

\end{document}